\documentclass[12pt]{article}
\usepackage[utf8]{inputenc}
\usepackage[a4paper,includeheadfoot,margin=2.54cm]{geometry} 
\pdfoutput=1

\usepackage[english]{babel}

\usepackage{amsmath,amsfonts,amsthm,amssymb}

\newcommand{\dx}{}

\usepackage{todonotes}
\usepackage{algorithm,algorithmicx,algpseudocode}
\usepackage{subcaption}
\usepackage{stmaryrd}
\usepackage{mathtools}
\usepackage[margin=10pt,font=small,labelfont=bf]{caption}
\usepackage{cite}
\usepackage[]{units}

\usepackage[plainpages=false]{hyperref}
\hypersetup{bookmarks=true,
		bookmarksopen=false,
		pdfborder={0 0 0},
		pdftitle={Shape Optimization Using the Cut Finite Element Method},
		pdfauthor={E. Burman, D. Elfverson, P. Hansbo, M.G. Larson, K. Larsson},
		pdfcreator={E. Burman, D. Elfverson, P. Hansbo, M.G. Larson, K. Larsson},
		colorlinks=true,
		linkcolor=blue,
		citecolor=red	
	}

\newtheorem{theorem}{Theorem}[section]

\theoremstyle{definition}

\newtheorem{remark}{Remark}[section]

\numberwithin{equation}{section}
\newcommand{\bndmap}{\mathcal{M}}
\newcommand{\mcL}{\mathcal{L}}

\newcommand{\mcO}{\mathcal{O}}
\newcommand{\IR}{\mathbb{R}}
\newcommand{\DtV}{D_{t,\vartheta}}
\newcommand{\DOV}{D_{\Omega,\vartheta}}
\newcommand{\POV}{\partial_{\Omega,\vartheta}}

\newcommand{\DJv}{D_\Omega J(\vartheta)}
\newcommand{\K}{\mathcal{K}}
\newcommand{\F}{\mathcal{F}}
\newcommand{\Omegamin}{\Omega_\mathrm{min}}

\DeclareMathOperator{\trace}{tr}
\DeclareMathOperator{\Div}{Div}

\newcommand{\ljump}{\llbracket}
\newcommand{\rjump}{\rrbracket}

\title{\bf Shape Optimization Using the\\ Cut Finite Element Method\thanks{This research was supported in part by the Swedish Foundation for Strategic Research Grant No.\ AM13-0029, the Swedish Research Council Grants Nos.\ 2011-4992, 2013-4708, the Swedish Research Programme Essence, and EPSRC, UK, Grant Nr. EP/J002313/1.}}

\author{Erik Burman\footnote{Department of Mathematics, University 
 College London, Gower Street, London WC1E 6BT, UK, e.burman@ucl.ac.uk}
 \mbox{ } 
Daniel Elfverson\footnote{Department of Mathematics and Mathematical Statistics, Ume{\aa} University, SE--901~87~~Ume{\aa}, Sweden, daniel.elfverson@umu.se}
 \mbox{ }
 Peter Hansbo\footnote{Department of Mechanical Engineering, J\"onk\"oping University, SE-551~11 J\"onk\"oping, Sweden, peter.hansbo@ju.se } 
\\
 \mbox{ } 
{Mats G. Larson}\footnote{Department of Mathematics and Mathematical Statistics, Ume{\aa} University, SE--901~87~~Ume{\aa}, Sweden, mats.larson@umu.se}
\mbox{ } 
{Karl Larsson}\footnote{Department of Mathematics and Mathematical Statistics, Ume{\aa} University, SE--901~87~~Ume{\aa}, Sweden, karl.larsson@umu.se}
}

\begin{document}
\maketitle
\begin{abstract}
We present a cut finite element method for shape optimization 
in the case of linear elasticity. The elastic domain is defined by 
a level-set function, and the  evolution of the domain is obtained by moving the level-set along a velocity field using a transport 
equation. The velocity field is the largest decreasing direction of the shape derivative that satisfies a certain  regularity requirement 
and the computation of the shape derivative is based on a volume formulation. Using the cut finite element method no re--meshing 
is required when updating the domain and we may also use higher 
order finite 
element approximations. To obtain a stable method, stabilization terms are added in the vicinity of the cut elements at the boundary, which provides control of the variation of the solution in the vicinity of the boundary. We implement and illustrate the performance of the method 
in the two--dimensional case, considering both triangular and quadrilateral meshes as well as finite element spaces of different 
order.
\end{abstract}

\clearpage
\tableofcontents
\clearpage

\section{Introduction} 


Optimization of elastic structures is an important and active research field  
of significant interest in engineering. There are two common approaches to 
represent the domain which we seek to optimize: (i) \emph{A density
  function.} This  approach, common in topology optimizaton
\cite{BeSi03, ChKl09}, is very general and computationally convenient, but the boundary representation is not sharp and thus typically fine grids and low order approximation spaces are employed. (ii) \emph{An implicit or explicit representation 
of the boundary.} This approach is common in shape optimization \cite{SZ92} where the  boundary is typically described by a level-set function or a parametrization, but topological changes can also be handled for instance using 
an implicit level-set representation of the boundary. Given the boundary representation  we need to generate a discretization of the domain when it is updated. This can be done using a standard meshing approach based on mesh motion and/or re-meshing or alternatively using a fictitious domain method, see \cite{AJT04,Sewi00,WaWa03} for different approaches.



In this contribution  we focus on the fictitious domain approach using the recently developed cut finite element method CutFEM \cite{BH12,BCHLM15}, extending 
our previous work  on the Bernoulli free boundary value problem \cite{BernoulliReport} to linear elasticity. The key components in CutFEM are: (i) Use of a fixed background mesh and a sharp boundary representation that is allowed to cut through the background mesh 
in arbitrary fashion. (ii) Weak enforcement of the boundary conditions. (iii)  Stabilization of the cut elements in the vicinity of the boundary using a 
consistent stabilization term which leads to optimal order accuracy and conditioning of the resulting algebraic system of equations. CutFEM also 
allows higher order finite element spaces and rests on a solid theoretical foundation including stability bounds, optimal order a priori error bounds, 
and optimal order bounds for the condition numbers of the stiffness and 
mass matrices, see \cite{BCHLM15} and the references therein. 

In order to support large changes in the shape and topology of the domain 
during the optimization process we employ a level-set representation of the boundary. The evolution of the domain is obtained by moving the level-set 
along a velocity field using a Hamilton-Jacobi  transport equation, see   \cite{ADF14, AJT04}. The velocity field is the largest decreasing direction 
of the shape derivative that satisfies a certain regularity requirement together 
with a boundary conditions on the boundary of the design volume. The
computation of the shape derivative is based on a volume formulation,
see \cite{HP15, HPS15} for similar approaches. In this context CutFEM provides 
an accurate and stable approximation of the linear elasticity  equations which completely avoids the use of standard meshing procedures when updating the domain. In this contribution we focus on standard Lagrange elements, but a wide 
range of elements may be used in CutFEM, including discontinuous elements, isogeometric elements with higher order regularity, and mixed elements. 

When using a higher order finite element space in shape optimization we use 
a finer representation of the domain than the computational mesh, i.e., the level-set representing the geometry is defined on a finer mesh. For computational convenience and efficiency we use a  piecewise linear geometry description 
on the finer grid. This approach of course leads to loss of optimal order but the main purpose of using the finer grid is to allow the domain to move more freely on the refined grid despite using larger higher order elements for the approximation of the solution field. If necessary, once a steady design has been found, a more accurate final computation can be done which uses a higher order geometry description. Our approach leads to a convenient and efficient implementation 
since the solution to the elasticity equations and the level-set function are represented using the fixed background mesh and a uniform refinement thereof.

An outline of the paper is as follows: in Section 2 we formulate the equations of 
linear elasticity and the optimization problem, in Section 3 we formulate the CutFEM, in Section 4 we recall the necessary results from shape calculus, in 
Section 5 we formulate the transport equation for the level-set and the 
optimization algorithm, and finally in Section 6 we present numerical results.

\label{sec:introduction}

\section{Model Problem} 
\label{sec:model_problem}
\subsection{Linear Elasticity}
Let $\Omega\in\mathbb{R}^d$, for $d=2,3$, be a bounded domain 
with boundary $\partial\Omega=\Gamma_D\cup\Gamma_N$ such that 
$\Gamma_D \cap \Gamma_N = \emptyset$, and let $n$ be the exterior 
unit normal to $\partial\Omega$. Assuming a linear elastic isotropic 
material the constitutive relationship between the symmetric stress 
tensor $\sigma$ and the strain tensor $\epsilon$ is given 
by Hooke's law
\begin{equation}
	\sigma = 2\mu\epsilon + \lambda\trace(\epsilon)I
\end{equation}
where $\mu$ and $\lambda$ are the Lam{\' e} parameters and $I$ is 
the $d\times d$ identity matrix. Also, assuming small strains we may use the 
linear strain tensor $\epsilon(u)=(\nabla u + \nabla u^T)/2$, 
where $u$ is the displacement field, as a strain measure. In this expression for $\epsilon(u)$ we take the gradient of vector fields which we define as $\nabla u = u\otimes\nabla$, i.e. the Jacobian of $u$. Under 
these assumptions the governing equations for the stress $\sigma$ 
and displacements $u$ of an elastic body in equilibrium are
\begin{alignat}{3}
    -\Div \sigma &= f&\qquad & \text{in }\Omega \label{eq:elasticity-a}\\
    \sigma &= 2\mu\epsilon(u) + \lambda\trace(\epsilon(u))I&\qquad& \text{in }\Omega \label{eq:elasticity-b}\\
    u &= 0 & \qquad & \text{on }\Gamma_D \label{eq:elasticity-c}\\
    \sigma\cdot n &= g &\qquad & \text{on }\Gamma_N \label{eq:elasticity-d}
\end{alignat}
where $\Div\sigma$ is the row-wise divergence on $\sigma$ and $f$ and $g$ are given body and surface force densities, respectively.

In shape optimization we seek to minimize some objective functional $J(\Omega;u(\Omega))$ with respect to the domain $\Omega$ where $u(\Omega)$ 
is the displacement field obtained by solving \eqref{eq:elasticity-a}--\eqref{eq:elasticity-d}. For the present work we choose to minimize the 
compliance, i.e. the internal energy of the elastic body, and for clarity of presentation we also choose $f=0$.

\paragraph{Weak Form.}
The weak form of the problem
\eqref{eq:elasticity-a}--\eqref{eq:elasticity-d} reads: find $u\in V(\Omega)=\{u\in[H^1(\Omega)]^d : u|_{\Gamma_D}=0\}$
such that
\begin{equation}\label{eq:primal}
	a(\Omega;u,q) = L(q)\qquad\forall q \in V(\Omega)
\end{equation}
with bilinear form $a(\Omega;\cdot,\cdot)$ and linear functional $L(\cdot)$ given by
\begin{align}
	a(\Omega;v,q) &= 2\mu(\epsilon(v),\epsilon(q))_{L^2(\Omega)} + \lambda(\nabla\cdot v,\nabla\cdot q)_{L^2(\Omega)} \\
	L(q) & = (g,q)_{L^2(\Gamma_N)}
\end{align}
where $(\cdot,\cdot)_{L^2(\Omega)}$ denotes the usual $L^2(\Omega)$ inner product.
Specifically, for tensors $A,B:\Omega\rightarrow\IR^{d\times d}$ the $L^2(\Omega)$ inner product is
$(A,B)_{L^2(\Omega)} = \int_\Omega A:B$ where
the contraction operator ":"  is defined $A:B = \sum_{i,j=1}^d A_{ij}B_{ij}$ and we recall that $\trace(\epsilon(v))=I:\epsilon(v)=\nabla\cdot v$.
Note that by choosing $f=0$  and $\Gamma_N$ fixed the linear functional $L$ is independent of $\Omega$.

\subsection{Minimization Problem}
Given some functional $J(\Omega;v)$ where $\Omega\subset\Omega_0$ is a domain and $v\in V(\Omega)$ we define the objective functional
\begin{equation} \label{eq:general-objective}
J(\Omega) = J(\Omega;u(\Omega))
\end{equation}
where $u(\Omega)\in V(\Omega)$ is the solution to \eqref{eq:primal}.
Letting $\mcO$ denote the set of all admissible domains we formulate the 
minimization problem: find $\Omegamin\in\mcO$ such that
\begin{align}\label{eq:minimization}
	J(\Omegamin) = \min_{\Omega\in\mcO} J(\Omega)
\end{align}

\paragraph{Lagrangian Formulation.}
The minimization problem can also be expressed as finding an extreme point
$(\Omegamin;u(\Omegamin),p(\Omegamin))$ to the Lagrangian
\begin{equation}\label{eq:lagrangian}
	\mcL(\Omega;v,q) =  J(\Omega;v) - a(\Omega;v,q)+L(q)
\end{equation}
and we note that $J(\Omega;v)=\mcL(\Omega;v,q)$ for all $q\in V$ and $\Omega\in\mcO$ by \eqref{eq:primal}.
For a fixed domain $\Omega$, the saddle point $(u,p)$ fulfilling
\begin{align}
\mcL(\Omega;u,p) = \adjustlimits \inf_{v\in V} \sup_{q\in V} \mcL(\Omega;v,q)
\label{eq:saddlepoint}
\end{align}
can be determined by the problems: find 
$u\in V(\Omega)$ and $p\in V(\Omega)$ such that
\begin{align}
	\langle\partial_q\mcL(\Omega;u,p),\delta q\rangle &= 0\qquad\forall \delta q \in V(\Omega)
	\label{eq:dLdq}
	\\
	\langle\partial_v\mcL(\Omega;u,p),\delta v\rangle &= 0\qquad\forall \delta v \in V(\Omega)
	\label{eq:dLdv}
\end{align}
where $\langle\partial_q\mcL(\Omega,u,p),\delta q\rangle$ denotes the partial derivative of $\mcL$ with respect to $q$ in the direction of $\delta q$, cf. \cite{AJT04}.
From \eqref{eq:dLdq} we can derive the primal problem \eqref{eq:primal}; thus $u$ is the solution to \eqref{eq:primal}, and
from \eqref{eq:dLdv} we deduce that $p$ is the solution to the following dual problem: find $p\in V(\Omega)$ such that
\begin{equation}\label{eq:dual}
	 a(\Omega;v,p) = \langle\partial_v J(\Omega,u),v\rangle\quad\forall v \in V(\Omega)
\end{equation}

Recalling that $J(\Omega;v) = \mcL(\Omega;v,q)$ we can thus express the objective functional \eqref{eq:general-objective} in terms of the Lagrangian as
\begin{equation}\label{eq:J-inf-sup}
J(\Omega) = \mcL(\Omega;u(\Omega),p(\Omega))
= \adjustlimits \inf_{v\in V(\Omega)} \sup_{q\in V(\Omega)} \mcL(\Omega;v,q) 
\end{equation}
where $u(\Omega)\in V(\Omega)$ and $p(\Omega)\in V(\Omega)$ are the solutions to
\begin{alignat}{2}
a(\Omega;u(\Omega),v) &=L(v)\qquad &&\forall v\in V(\Omega)
\\
a(\Omega;v,p(\Omega)) &= \langle\partial_v J(\Omega;u(\Omega)),v\rangle\qquad &&\forall v \in V(\Omega)
\end{alignat}
i.e., the solutions to the primal and dual problem, respectively.

\paragraph{Objective Functional.}
In the present work we choose to minimize the compliance, i.e. the elastic energy, which is given by
\begin{align}
\frac{1}{2}a(\Omega;u(\Omega),u(\Omega))
\end{align}
where $u(\Omega)$ is the solution to \eqref{eq:primal}.
Since we also want to constrain the volume $|\Omega|$
we form the functional $J(\cdot;\cdot)$ as the sum of the compliance and a
penalty term
\begin{equation}
	J(\Omega;v) = \frac{1}{2}a(\Omega;v,v) + \kappa |\Omega|
\label{eq:compliance-J}
\end{equation}
where $\kappa$ acts as a material cost which we in the present work keep fixed.
For this choice of objective functional the dual problem \eqref{eq:dual} coincides with the primal problem \eqref{eq:primal} since $\langle\partial_v J(\Omega;u),v\rangle=L(v)$ and the bilinear form is symmetric, and hence $p=u$.
\begin{remark}
If a fixed amount of material is desired, i.e., $|\Omega|=\gamma|\Omega_0|$ for
some fixed $0<\gamma<1$, it is possible to determine $\kappa$ by some adaptive
update strategy of the material cost $\kappa$.
\end{remark}

\begin{remark}\label{rem:dual}
	When considering other functionals $J(\Omega;v)$,
	typically $u\neq p$ and the dual problem \eqref{eq:dual} must be computed.
\end{remark}


\section{Cut Finite Element Method} 
\label{sec:CutFEM}
We will use a cut finite element method to solve the primal (and dual) equation \eqref{eq:primal}. An analysis of this method for linear elasticity is presented in \cite{cutfem-elasticity}.

\subsection{The Mesh and Finite Element Spaces}
Consider a fixed polygonal domain $\Omega_0$, with $\Omega\subset\Omega_0$ as illustrated in Figure~\ref{fig:mesh1-domain}, and let 
\begin{equation}
	\K_{h,0}=\{K\}\quad\text{and}\quad\F_{h,0}=\{F\}
\end{equation}
denote a subdivision of $\Omega_0$ into a family of quasi-uniform triangles/tetrahedrons
or a uniform quadrilaterals/bricks with mesh parameter $h \in (0,h_0]$, illustrated in Figure~\ref{fig:mesh2-grid}, respectively the set 
of interior faces in $\K_{h,0}$. Let $P_k$ be the space of full polynomials in $\mathbb{R}^d$ up to degree $k$ on triangular/tetrahedral elements and tensor 
products of polynomials up to degree $k$ on quadrilateral/brick elements. We 
define the finite element space of continuous piecewise polynomial functions on $\K_{h,0}$ by
\begin{equation}
V_{h,k,0}=\{v\in V(\Omega_0) \cap [C^0(\Omega_0)]^d \,:\, v|_K\in [P_k(K)]^d \, , \, K\in\K_{h,0}\}
\end{equation}
The active part of the mesh is given by all elements in $\K_{h,0}$ which has a non-zero intersection
with the domain $\Omega$.
We define the active mesh and its interior faces
\begin{equation}
	\K_h = \{K\in \K_{h,0}:  \overline{K} \cap\Omega \neq \emptyset \}\quad\text{and}\quad \F_h = \{F\in \F_{h,0}:  F\cap\Omega \neq \emptyset\}
\end{equation}
the union of all active elements
\begin{equation}
	\Omega_h = \cup_{K\in\K_h} K
\end{equation}
illustrated in Figure~\ref{fig:mesh3-active},
and the finite element space on the active mesh
\begin{equation}
	V_{h,k}(\Omega_h)=V_{h,k,0}|_{\Omega_h}
\end{equation}
We also define the sets of all elements that are cut by $\Gamma_D$ and $\Gamma_N$, respectively
\begin{align}
	\Omega_{h,D} &= \{\cup K: \overline{K} \cap \Gamma_D \neq\emptyset,\, K\in \K_h\} 
	\\
	\Omega_{h,N} &= \{\cup K: \overline{K} \cap \Gamma_N \neq\emptyset,\, K\in \K_h\} 
\end{align}
and the sets of interior faces belonging to elements in $\K_h$ that are cut by $\Gamma_D$ respectively cut by $\Gamma_N$ but not by $\Gamma_D$
\begin{align}
\F_{h,D} &= \{F\in \F_{h}:  F\cap\overline{\Omega}_{h,D} \neq \emptyset\}
\\
\F_{h,N} &= \{F\in \F_{h}\setminus\F_{h,D} :  F\cap\overline{\Omega}_{h,N} \neq \emptyset\}
\end{align}
which are illustrated in Figure~\ref{fig:mesh4-stab}. Note that $\F_{h,D}\cap\F_{h,N}=\emptyset$.
For each face $F \in \F_h$ we choose to denote one of the two elements sharing $F$ by $K_+$ and the other element by $K_-$. We thus have $F=\overline{K}_+\cap\overline{K}_-$ and we define the face normal $n_F= n_{\partial K_+}$ and the jump over the face
\begin{align}
\ljump v \rjump &= v|_{K_+} - v|_{K_-}
\end{align}

\begin{figure}
\centering
\begin{subfigure}[b]{0.35\textwidth}\centering
\includegraphics[width=0.87\linewidth]{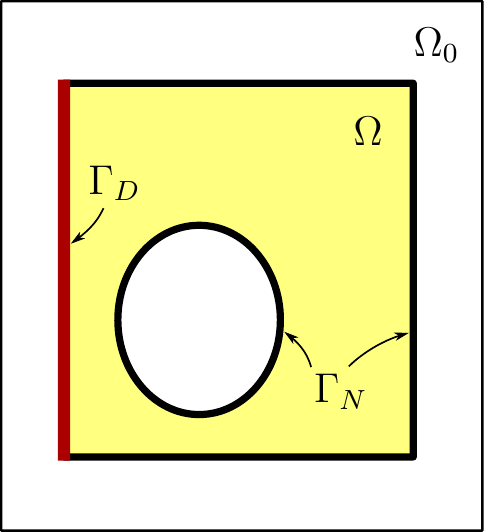}
\caption{Domain and boundaries}
\label{fig:mesh1-domain}
\end{subfigure}
\begin{subfigure}[b]{0.35\textwidth}\centering
\includegraphics[width=0.87\linewidth]{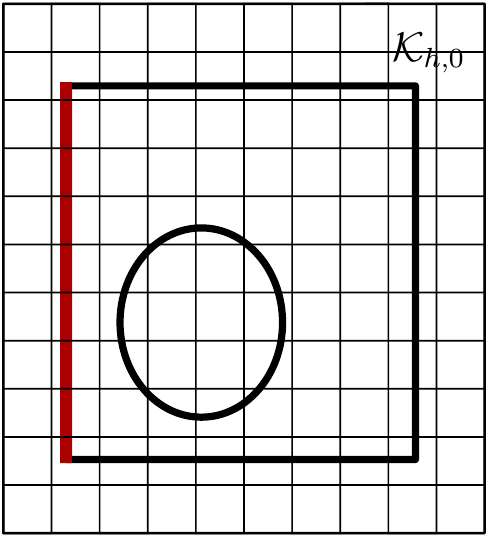}
\caption{Background grid}
\label{fig:mesh2-grid}
\end{subfigure}

\vspace{1.0em}

\begin{subfigure}[b]{0.35\textwidth}\centering
\includegraphics[width=0.87\linewidth]{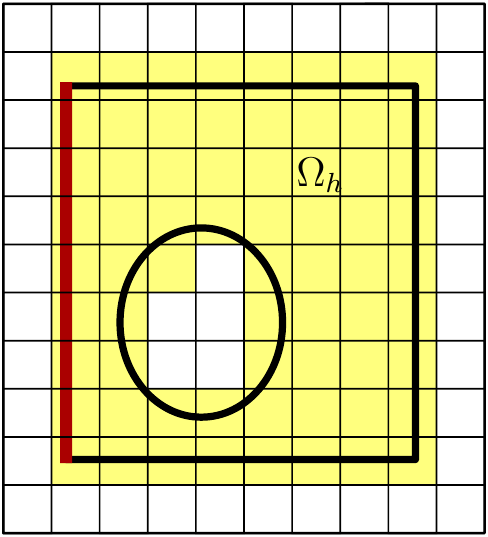}
\caption{Domain of active elements $\K_h$}
\label{fig:mesh3-active}
\end{subfigure}
\begin{subfigure}[b]{0.35\textwidth}\centering
\includegraphics[width=0.87\linewidth]{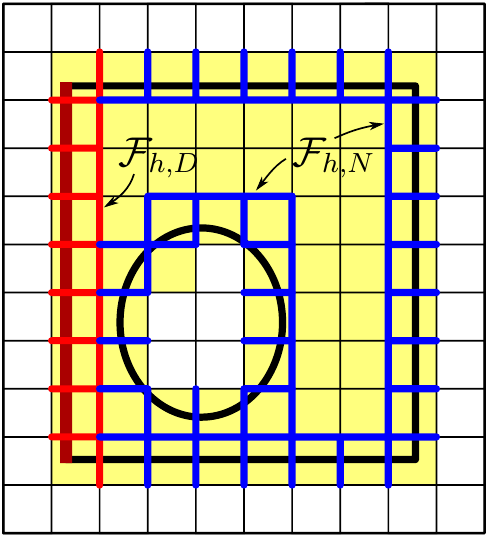}
\caption{Sets of faces}
\label{fig:mesh4-stab}
\end{subfigure}
\caption{Illustration of notation for meshes and faces.}
\label{fig:meshes}
\end{figure}

\subsection{The Method}
Following the procedure in \cite{cutfem-elasticity} we
introduce the stabilized bilinear form
\begin{align}
	a_h(\Omega;u,v) &= a(\Omega;u,v) + s_h(\F_{h,D};u,v) + h^2 s_h(\F_{h,N};u,v)
\end{align}
The stabilization form $s_h$ is given by
\begin{equation}\label{eq:stab}
s_h(\F;u,v) = \sum_{F \in \F} 
\sum_{j=1}^k \gamma_j h^{2j-1} (\ljump \partial_{n_F}^j u \rjump,\ljump \partial_{n_F}^j v \rjump)_{L^2(F)}
\end{equation}
where $\partial_{n_F}^j$ denotes the $j$:th derivative in the direction of the face normal $n_F$ and $\{\gamma_j\}_{j=1}^k$ are positive parameters.
We also introduce the stabilized Nitsche form
\begin{align}\label{eq:nitschesform}
	A_h(\Omega;u,v) &= a_h(\Omega;u,v)
-(\sigma(u)\cdot n ,v)_{L^2(\Gamma_{D})} -(u, \sigma(v)\cdot n)_{L^2(\Gamma_{D})}
\\ \nonumber &\qquad
+ \gamma_D h^{-1} \left(
2\mu (u,v)_{L^2(\Gamma_{D})}
+ \lambda (u\cdot n,v\cdot n)_{L^2(\Gamma_{D})}
\right)
\end{align}
where the additional terms give the weak enforcement of the Dirichlet boundary conditions via Nitsche's method \cite{Nitsche71} and $\gamma_D>0$ is a parameter.

The cut finite element method for linear elasticity can now be formulated as the following problem: find $u_h \in V_{h,k}(\Omega_h)$ such that
\begin{equation} \label{eq:method}
	A_h(\Omega;u_h,v) = L(v)\qquad\forall v\in V_{h,k}(\Omega_h)
\end{equation}

\paragraph{Theoretical Results.}
To summarize the the main theoretical results from \cite{cutfem-elasticity} the cut finite element method for linear elasticity has the following properties:
\begin{itemize}
\item The stabilized form $a_h$ enjoys the same coercivity and continuity properties with respect to the proper norms on $\Omega_h$ as the standard form $a$ on $\Omega$ equipped with a fitted mesh.
\item Optimal order approximation holds in the relevant norms since there is a  stable interpolation operator with an extension operator that in a $H^s$ stable way extends a function from $\Omega$ to a neighborhood of $\Omega$ containing $\Omega_h$.
\end{itemize}
Using these results a priori error estimates of optimal order can be derived using standard techniques of finite element analysis.


\subsection{Geometry Description} \label{sec:geom}

\begin{figure}
\centering
\begin{subfigure}[b]{0.4\textwidth}\centering
\includegraphics[width=0.30\linewidth]{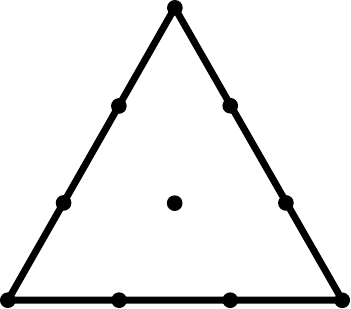}
\qquad
\includegraphics[width=0.30\linewidth]{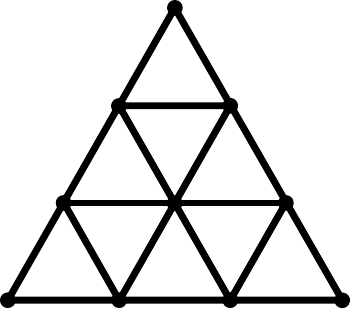}
\caption{Triangles}
\end{subfigure}
\begin{subfigure}[b]{0.4\textwidth}\centering
\includegraphics[width=0.30\linewidth]{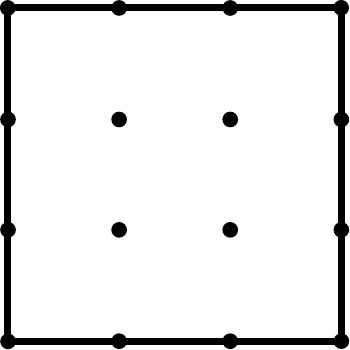}
\qquad
\includegraphics[width=0.30\linewidth]{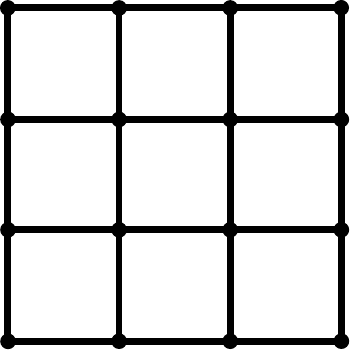}
\caption{Quadrilaterals}
\end{subfigure}
\caption{$P_3$ Lagrange elements and the corresponding mesh for a $P_1$-iso-$P_3$ finite element space consisting of $P_1$ Lagrange elements with vertices at the positions of the $P_3$ Lagrange nodes.}
\label{fig:lagrange-nodes}
\end{figure}

\begin{figure}
\centering
\begin{subfigure}[b]{\textwidth}\centering
\includegraphics[width=0.6\linewidth]{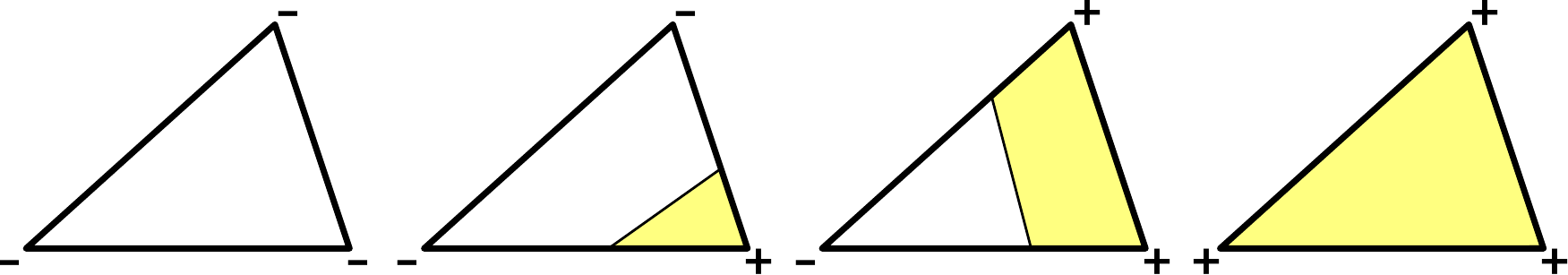}
\caption{Triangles}
\end{subfigure}

\vspace{1em}

\begin{subfigure}[b]{\textwidth}\centering
\includegraphics[width=0.6\linewidth]{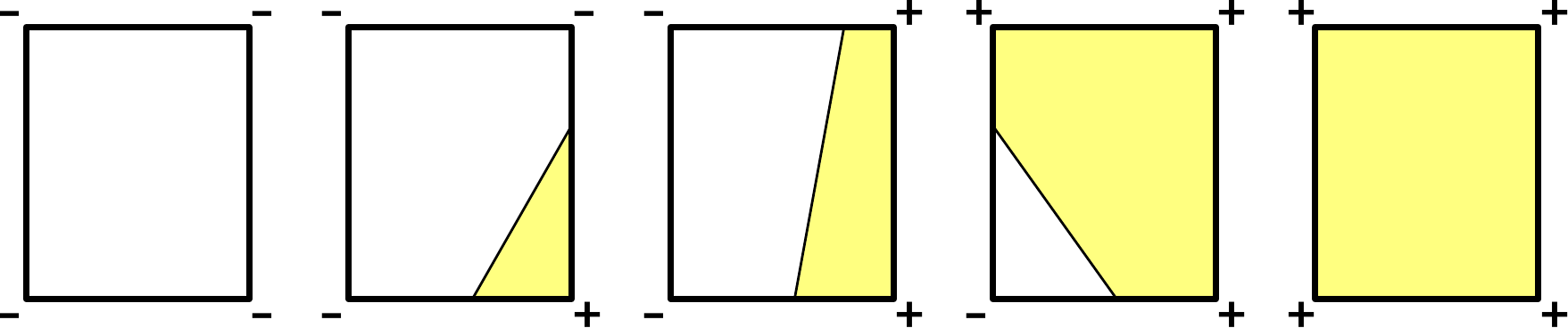}
\caption{Quadrilaterals}
\end{subfigure}
\caption{Cases of extracting piecewise linear geometry from a level-set on an element $K\in\mathcal{K}_{h/k,0}$ with the signs of the level-set at the vertices as indicated.}
\label{fig:level-set-cases}
\end{figure}

\begin{figure}
\centering
\begin{subfigure}[b]{0.4\textwidth}\centering
\includegraphics[width=0.87\linewidth]{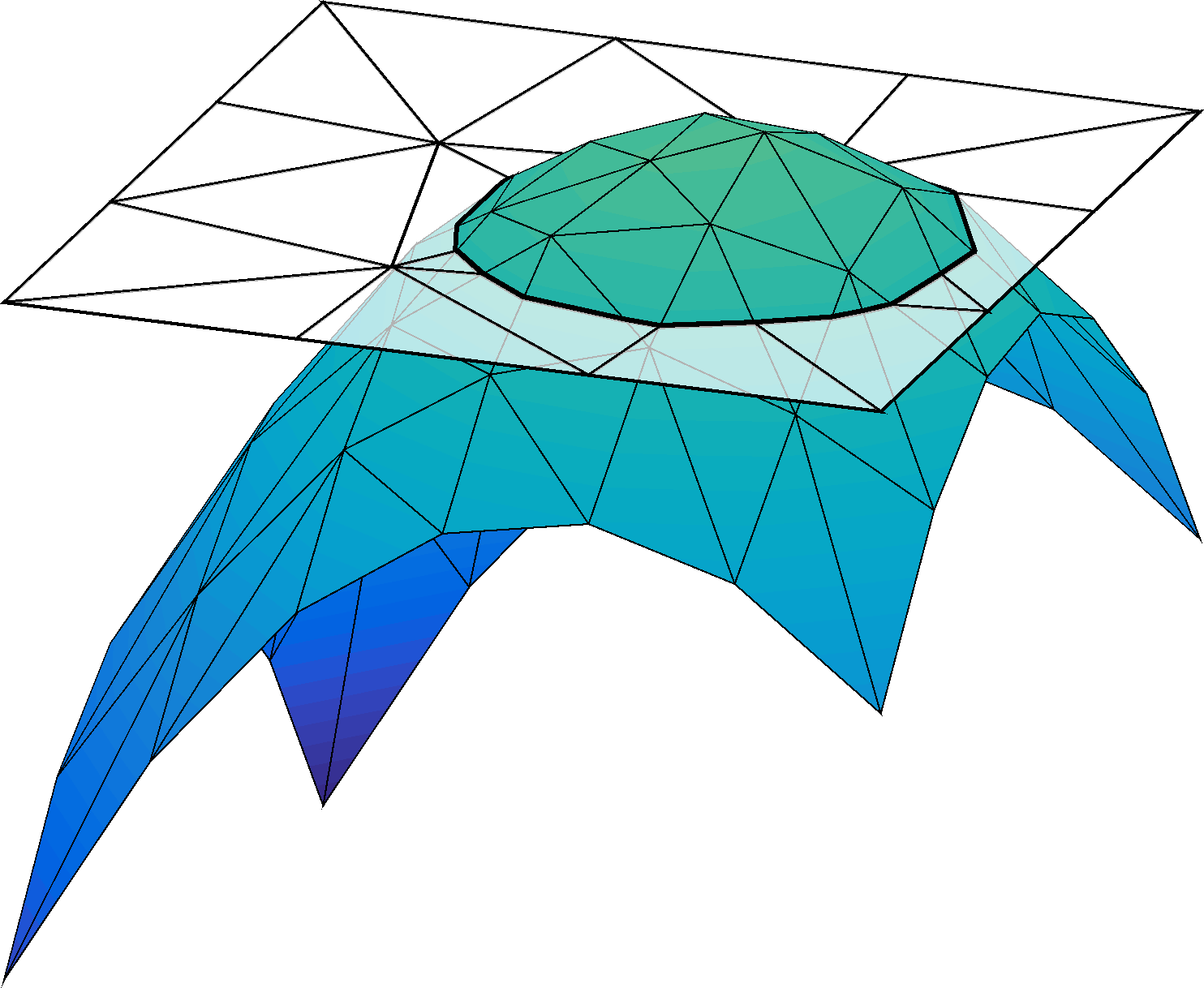}
\caption{$P_1\text{-iso-}P_2$ level-set}
\label{fig:levelset-tri}
\end{subfigure}
\begin{subfigure}[b]{0.30\textwidth}\centering
\includegraphics[width=0.87\linewidth]{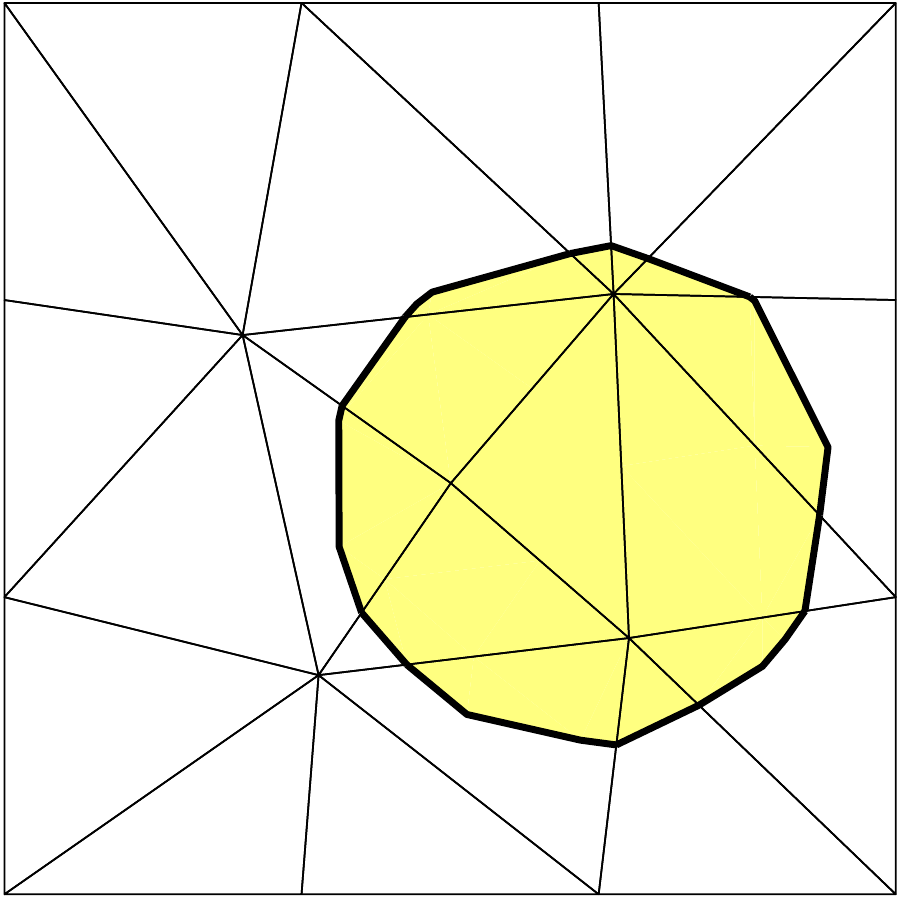}

\vspace{1.2em}

\caption{$P_1\text{-iso-}P_2$ geometry}
\label{fig:levelset-tri-geom}
\end{subfigure}

\vspace{1.0em}

\begin{subfigure}[b]{0.4\textwidth}\centering
\includegraphics[width=0.87\linewidth]{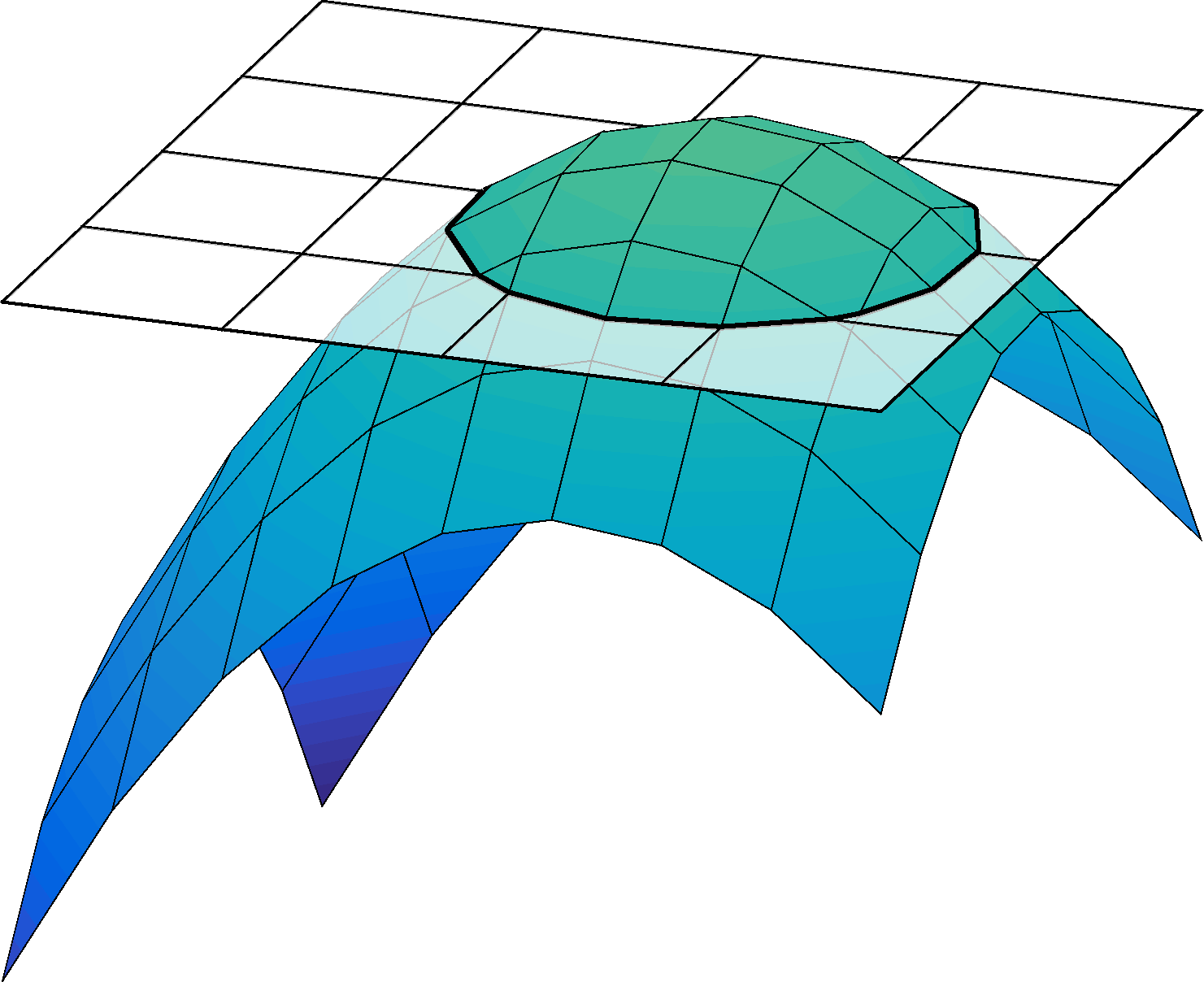}
\caption{$P_1\text{-iso-}P_2$ level-set}
\label{fig:levelset-quad}
\end{subfigure}
\begin{subfigure}[b]{0.30\textwidth}\centering
\includegraphics[width=0.87\linewidth]{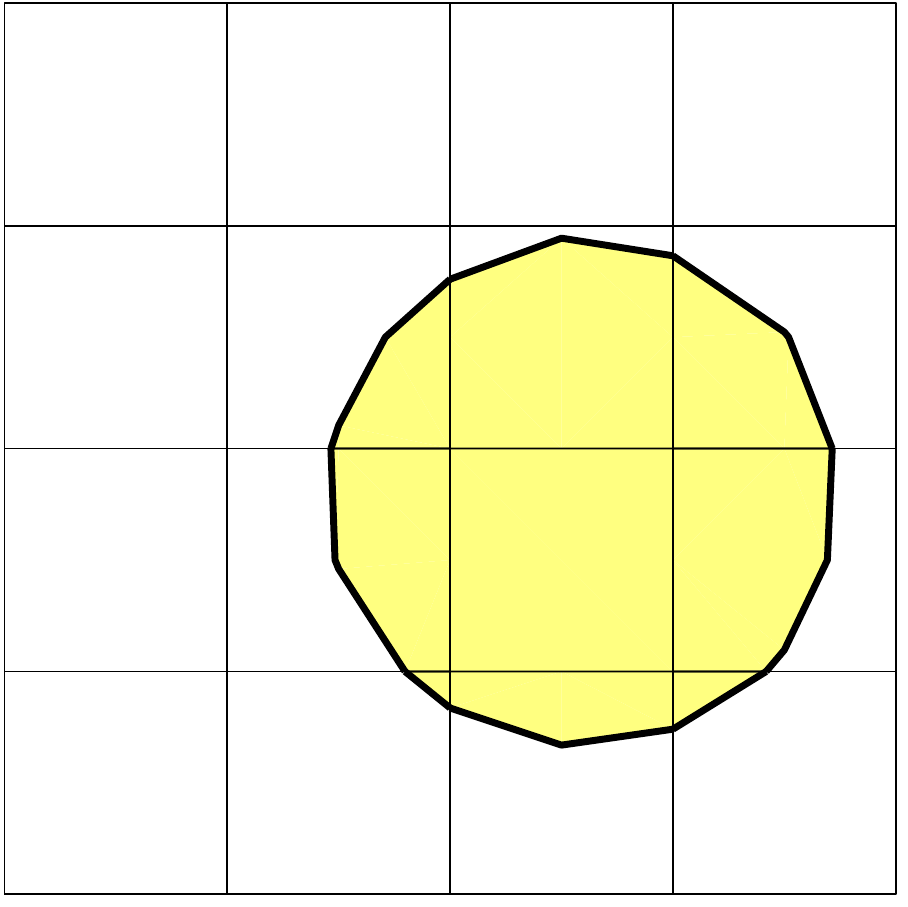}

\vspace{1.2em}

\caption{$P_1\text{-iso-}P_2$ geometry}
\label{fig:levelset-quad-geom}
\end{subfigure}
\caption{Piecewise linear geometry represented by a discrete level-set using $P_1\text{-iso-}P_k$ finite elements.}
\label{fig:levelset-geom}
\end{figure}

Let the geometry be described via a level-set function $\phi(x):\Omega_0 \rightarrow \mathbb{R}$ where the domain $\Omega$ and the domain boundary $\partial\Omega$ are given by
\begin{equation}
\Omega = \{\, x\in\Omega_0 \, : \, \phi(x)>0 \,\}
\qquad\text{and}\qquad
\partial\Omega = \{\, x\in\overline{\Omega}_0 \, : \, \phi(x)=0 \,\}
\end{equation}
For convenience we use $P_1$ finite elements for the level-set representation. However, as we use $P_k$ finite elements to approximate the solution, we improve the geometry representation when using higher order finite elements ($k\geq 2$) by letting the level-set be defined on a refined mesh.
The refined mesh $\K_{h/k,0}$ is constructed by uniform refinement of $\K_{h,0}$ such that the Lagrange nodes of a $P_k$ finite element in $\K_{h,0}$ coincides with the vertices of $\K_{h/k,0}$, as illustrated in Figure~\ref{fig:lagrange-nodes}.
The finite element space of the level-set is the scalar valued so called $P_1\text{-iso-}P_k$ finite element space on $\K_{h,0}$,
defined as
\begin{align}
W_{h/k}=\{v\in H^1(\Omega_0)\cap C^0(\Omega_0) \,:\,v|_K\in P_1(K),\
 K\in\K_{h/k,0}\}
\end{align}

\paragraph{Geometry Extraction.}
We extract the domain $\Omega$ from the level-set function $\phi\in W_{h/k}$ by traversing all elements in $\K_{h/k,0}$ and checking the value and sign of $\phi$ in the element vertices to derive the intersection between the element and the domain. This procedure results in a number of simple cases which we for triangles and quadrilaterals illustrate in Figure~\ref{fig:level-set-cases}.
Note that in the case of quadrilaterals the boundary intersection $\partial\Omega\cap K$ where $K\in\K_{h/k,0}$ is actually not a linear function as bilinear basis functions include a quadratic cross term, but we employ linear interpolation between detected edge intersections to produce the $P_1$ boundary illustrated.
Example geometry extractions from $P_1$-iso-$P_2$ level-sets on triangles and quadrilaterals are shown in Figure~\ref{fig:levelset-geom}.

\paragraph{Quadrature.} As the possible geometry intersections with elements in the refined mesh $\K_{h/k,0}$ consists of a small number of cases, as illustrated in Figure~\ref{fig:level-set-cases},
we construct quadrature rules for exact integration of products of $P_k$ polynomials for each case.


\section{Shape Calculus} 
\label{sec:shape_derivative}

\subsection{Definition of the Shape Derivative}

\begin{figure}\centering
\qquad\qquad
\includegraphics[width=0.35\linewidth]{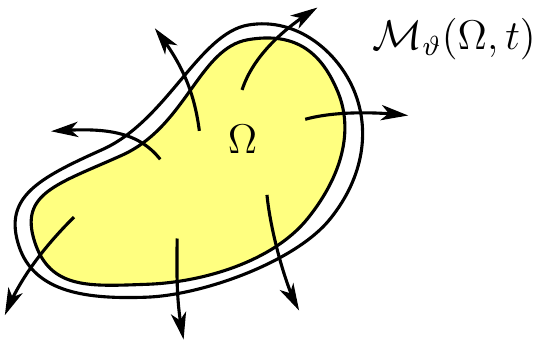}
\caption{Illustration of the mapping $\bndmap_\vartheta(\Omega,t)$.}
\label{fig:shape-derivative}
\end{figure}

For $\Omega \in \mcO$ we let $\mathcal{W}(\Omega,\IR^d)$ denote the space of sufficiently
smooth vector fields on $\Omega$ and for a vector field $\vartheta\in \mathcal{W}$ we define the mapping
\begin{equation}
  \bndmap_\vartheta:\Omega \times I \ni (x,t)  \mapsto x + t \vartheta(x) \in \bndmap_\vartheta(\Omega,t) \subset \IR^d 
\end{equation}
where $I = (-\delta,\delta)$, $\delta>0$. This mapping is illustrated in Figure~\ref{fig:shape-derivative}.
For small enough $\delta$, 
the mapping $\Omega \mapsto \bndmap_\vartheta(\Omega,t)$ is a bijection and $\bndmap_\vartheta(\Omega,0) = \Omega$. We also assume that the vector field 
$\vartheta$ is such that $\bndmap_\vartheta(\Omega,t) \in \mcO$ for $t\in I$ with $\delta$ 
small enough.

Let $J(\Omega)$ be a shape functional, i.e. a mapping $J:\mcO\ni\Omega\mapsto J(\Omega) \in \IR$. We then have the 
composition 
$I \ni t \mapsto J \circ \bndmap(\Omega,t) \in \IR$ and we define the shape 
derivative $\DOV$ of $J$ in the direction $\vartheta$ by
\begin{equation}
    \DOV J(\Omega) 
    = 
    \frac{d}{dt} J\circ \bndmap_\vartheta(\Omega,t)|_{t=0}
    =
    \lim_{t \to 0}\frac{J(\bndmap_\vartheta(\Omega,t))-J(\Omega)}{t}
\end{equation}
Note that if  $\bndmap_\vartheta(\Omega,t) = \Omega$ we have  $\DOV J = 0$,
even if $\bndmap_\vartheta(\omega,t) \neq \omega$ for a proper subspace
$\omega\subset\Omega$. We finally define the shape derivative $D_\Omega J:
W(\Omega,\IR^d) \rightarrow \IR$ by
\begin{equation}
 \DJv  = \DOV J(\Omega)
\end{equation}
If the functional $J$ depends on other arguments we use $\partial_\Omega$ to
denote the partial derivative with respect to $\Omega$ and $\POV$ to denote
the partial derivative with respect to $\Omega$ in the direction $\vartheta$.

For  $f:\Omega \times I  \rightarrow \IR^d$ we also define the material time 
derivative in the direction $\vartheta$ by
  \begin{equation}
         \DtV f =  \lim_{t\to 0}\frac{f(\bndmap_\vartheta(x,t),t)-f(x,0)}{t}
  \end{equation}
and recall the identity
\begin{equation}\label{eq:shape-id}
	    \POV\int_{\Omega} f\dx = \int_\Omega \DtV f + (\nabla\cdot \vartheta)f
\end{equation}
see for example \cite{BernoulliReport}.
	
\subsection{Shape Derivative}
We want to take the derivative of the inf-sup formulation of the objective
function \eqref{eq:J-inf-sup} with respect to the domain. The Correa--Seeger
theorem \cite{CS85} or \cite{DeZo11} states that for $q,v\in V(\Omega)$ we
have
\begin{equation}
	D_\Omega J(\vartheta) = \POV \mcL(\Omega; q, v)|_{v=u(\Omega),q=p(\Omega)}
\end{equation}
where  $u(\Omega)$ is the primal solution
\eqref{eq:primal}, and $p(\Omega)$ is the dual solution
\eqref{eq:dual}. For the special case when $u(\Omega)=p(\Omega)$ we have
\begin{equation}
	D_\Omega J(\vartheta) = \POV \mcL(\Omega; v, v)|_{v=u(\Omega)}
\end{equation} For notational simplicity we below omit the dependency on
$\Omega$, i.e., $u=u(\Omega)$ and $p=p(\Omega)$.

\begin{theorem}
	The shape derivative of the compliance objective functional $J(\Omega)=\mcL(\Omega;u,u)$ defined via \eqref{eq:compliance-J} is given by
	\begin{align}\label{eq:DmcL}
		\DOV J(\vartheta)
		&=  
		 \int_\Omega 2\mu(\epsilon_\vartheta(u):\epsilon(u)) + \lambda\mathrm{tr}(\epsilon_\vartheta(u))\mathrm{tr}(\epsilon(u))
		 \\ \nonumber & \qquad
		+ (\nabla\cdot\vartheta)
		\Big(\kappa-\mu(\epsilon(u):\epsilon(u)) - \frac{1}{2}\lambda\trace(\epsilon(u))^2 \Big) \notag 
	\end{align}
	where $\epsilon_\vartheta(u)$ is defined as
	\begin{equation}
		\epsilon_\vartheta(u) = \frac{1}{2}\left(\nabla u \nabla\vartheta + \nabla\vartheta^T \nabla u^T\right)
	\end{equation}
\end{theorem}
\begin{proof}
We have
\begin{align}
\POV \mcL(\Omega;v,v) &= \POV \left(\kappa|\Omega|-\frac{1}{2}a(\Omega;v,v)\right)
\end{align}
from \eqref{eq:lagrangian} and \eqref{eq:compliance-J} and by \eqref{eq:shape-id}
we obtain
\begin{align}
\POV \mcL\left(\kappa|\Omega|-\frac{1}{2}a(v,v)\right) &= \POV \int_\Omega  \kappa-\mu \epsilon(v):\epsilon(v) - \frac{1}{2}\lambda\trace(\epsilon(v))^2
\\&=
\int_\Omega \DtV \left( -\mu \epsilon(v):\epsilon(v) - \frac{1}{2}\lambda\trace(\epsilon(v))^2 \right)
\\&\qquad\nonumber
+ (\nabla\cdot \vartheta) \left(\kappa -\mu \epsilon(v):\epsilon(v) - \frac{1}{2}\lambda\trace(\epsilon(v))^2 \right)
\end{align}
where we used $\DtV\kappa=0$.
We introduce the compact notation $\bndmap_t=\bndmap_\vartheta(\cdot,t)$ for the mapping and $\Omega_t=\bndmap_t\Omega=\bndmap_\vartheta(\Omega,t)$ for the perturbed domain.
Letting $\nabla_t$ denote the gradient on $\Omega_t$ we have the identity
\begin{align}
\nabla_t (v \circ \bndmap_t^{-1})\circ \bndmap_t &= \nabla v (\nabla\bndmap_t)^{-1}
\end{align}
and thus $\epsilon(v)$ can be parametrized by $t$ as
\begin{align}
\epsilon_t(v)
&=
\frac{1}{2}\left(\nabla_t (v \circ \bndmap_t^{-1}) + (v \circ \bndmap_t^{-1})^T \nabla_t v ^T \right) \circ \bndmap_t
\\&=
\frac{1}{2}\left(\nabla v (\nabla\bndmap_t)^{-1} + (\nabla\bndmap_t)^{-T}\nabla v^T\right)
\end{align}
Using that
\begin{align}
&\DtV \nabla v  (\nabla\bndmap_t)^{-1}
= \nabla v  \DtV (\nabla\bndmap_t)^{-1} = -\nabla v \nabla\vartheta
\end{align}
we obtain
\begin{align}
\DtV \epsilon_t(v)
&= 
\DtV \frac{1}{2}\left(\nabla v \nabla\bndmap^{-1}_t + \nabla\bndmap^{-T}_t\nabla v^T\right) \\
& = - \frac{1}{2}\left(\nabla v \nabla\vartheta + \nabla\vartheta^T\nabla v^T\right) \\
& = - \epsilon_\vartheta(v)
\end{align}
and as a result
\begin{align}
\DtV \left( -\mu \epsilon(v):\epsilon(v) - \frac{1}{2}\lambda\trace(\epsilon(v))^2 \right)
&=
2\mu \epsilon_\vartheta(v):\epsilon(v) + \lambda\trace(\epsilon_\vartheta(v))\trace(\epsilon(v))
\end{align}
which concludes the proof by setting $v=u$.
\end{proof}

\subsection{Finite Element Approximation of the Shape Derivative} 
\label{sub:discrete_shape_derivative}

We compute a discrete approximation to the shape derivative
by inserting the cut finite element approximation $u_h$ to \eqref{eq:primal}, defined by \eqref{eq:method}, into \eqref{eq:DmcL}. This yields the following expression for the shape derivative approximation
\begin{align}\label{eq:discreteDcmL}
	\POV\mcL(\Omega;u_h,u_h)
		&=  
		 \int_\Omega 2\mu(\epsilon_\vartheta(u_h):\epsilon(u_h)) + \lambda\mathrm{tr}(\epsilon_\vartheta(u_h))\mathrm{tr}(\epsilon(u_h)) \\
		&
		\qquad +(\nabla\cdot\vartheta)\Big(\kappa - \mu(\epsilon(u_h):\epsilon(u_h)) - \frac{1}{2}\lambda(\nabla\cdot u_h)(\nabla\cdot u_h)\Big) \notag 
\end{align}

\section{Domain Evolution}
To construct a robust method for evolving the domain we need the discrete level-set function $\phi\in W_{h/k}$ to be a good approximation to a signed distance function, at least close to the boundary, i.e.,
\begin{equation}
	\phi(x) \approx
	\begin{cases}
		\phantom{-}\rho(x,\partial\Omega),\quad x\not\in\Omega \\
		-\rho(x,\partial\Omega),\quad x\in\Omega
	\end{cases}
\end{equation}
where $\rho(x,\partial\Omega)$ is the smallest Euclidean distance between 
the point $x$ and the boundary $\partial\Omega$.
Note that a property of a signed distance function $\phi$ is that $|\nabla\phi|=1$.
In the following Sections we 
consider the reinitialization of $\phi$ to a signed distance function 
and formulate a method to evolve $\phi$ using the transport equation
\begin{equation}
\partial_t \phi + \beta\cdot\nabla\phi = 0
\end{equation}
where $\beta$ is a velocity field which we compute based on the shape derivative.

\subsection{Reinitialization}
We consider so called Elliptic reinitialization to make the level-set to 
resemble a distance function together with a best approximation 
$L^2$ projection on the interface. 
The reinitialization is performed in two steps:
\begin{enumerate}
\item On the subdomain given by elements in $\K_{h/k,0}$ cut
by the boundary, i.e.
\begin{equation}
\Omega_\partial = \{\cup K: \overline{K} \cap \partial\Omega \neq\emptyset,\, K\in \K_{h/k,0}\} \, ,
\end{equation}
we construct the best distance function 
in $L^2$ sense by solving the problem: find $\phi_\partial\in
W_{h/k}|_{\Omega_\partial}$ such that
\begin{equation}
	(\phi_\partial,w)_{L^2(\Omega_\partial)}=(\phi/|\nabla\phi|,w)_{L^2(\Omega_\partial)}\qquad \forall w \in W_{h/k}|_{\Omega_\partial}
\end{equation}

\item On the rest of the domain, i.e. $\Omega_0\setminus\Omega_\partial$, we use an energy minimization
technique where we seek $\phi \in \{v \in W_{h/k} : v|_{\Omega_\partial} = \phi_\partial \}$ which minimizes
\begin{equation}
 \frac{1}{2}(1-|\nabla\phi|,1-|\nabla\phi|)_{L^2(\Omega_0)} 	
\end{equation}
i.e. makes $\phi$ resemble a distance function as closely as possible.
The minimization problem is equivalent to solving the non-linear problem: find $\phi \in \{v \in W_{h/k} : v|_{\Omega_\partial} = \phi_\partial \}$ such that
\begin{equation}
(\nabla \phi, \nabla v)_{L^2(\Omega_0)} = (|\nabla \phi|^{-1}\nabla \phi, \nabla v)_{L^2(\Omega_0)}
\qquad\forall v \in \{v \in W_{h/k} : v|_{\Omega_\partial} = 0 \}
\end{equation}
This can for example be accomplished by using the fixed point iteration
\begin{equation}
	(\nabla \phi^m, \nabla v)_{L^2(\Omega_0)} = (|\nabla
        \phi^{m-1}|^{-1} \nabla \phi^{m-1}, \nabla v)_{L^2(\Omega_0)}
\end{equation}

\end{enumerate}

\subsection{Shape Evolution} 
\label{sec:shape_evolution}

\paragraph{Computing the Velocity Field $\boldsymbol \beta$.}
Consider the bilinear form
\begin{equation}\label{eq:b_h}
	b(\beta,\vartheta) = (\beta,\vartheta)_{L^2(\Omega_0)} + c_1 (\nabla\beta , \nabla\vartheta)_{L^2(\Omega_0)}
\end{equation}
where $c_1 > 0$ is a parameter used for setting the amount of regularization.
We want to find the velocity field $\beta$ that satisfy
\begin{equation}\label{eq:decent}
	\min_{\sqrt{b(\beta,\beta)}=1} D_{\Omega,\beta} J(\Omega)
\end{equation}
This is equivalent to solving: find $\beta'\in [W_{h/k}]^d$ such that
\begin{equation}
	b(\beta',\vartheta) = - \DOV J(\Omega)\qquad\forall \vartheta \in [W_{h/k}]^d
\end{equation}
and set
\begin{equation}
	\beta = \beta'/\sqrt{b(\beta',\beta')}
\end{equation}
see for example \cite{BernoulliReport}.
As boundary conditions on $\beta'$ we use
\begin{equation}
\beta'\cdot n = 0
\quad\text{and}\quad
(I - n \otimes n) (\nabla\beta')n = 0
\qquad\text{on $\partial\Omega_0$}
\end{equation}
where $(I - n \otimes n)$ is the projection onto the tangential plane of the boundary.


\paragraph{Evolving the Level-Set $\boldsymbol\phi$.}
The evolution of the domain $\Omega$ over a pseudo-time step $T$ is computed by solving the following convection 
equation
\begin{align}\label{eq:transport_stab}
	(\partial_t\phi,v)_{L^2(\Omega_0)} +
  (\beta\cdot\nabla\phi,v)_{L^2(\Omega_0)} + c_2\sum_{F\in
  \F_h} h^2 (\llbracket \partial_{n_F} \phi \rrbracket,\llbracket\partial_{n_F} v\rrbracket)_{L^2(F)}&= 0,\quad t\in(0,T]+t_0
\end{align}
where $c_2$ is a stabilization parameter.
For time integration we use a Crank--Nicolson method and the time step $T$ is chosen via the optimization algorithm described in the next section.

\subsection{Optimization Algorithm} 
\label{sec:optimization_algorithm}
In this section we propose an algorithm that solve \eqref{eq:minimization} and
give a overview of the optimization procedure. To find the descent direction of
the boundary we first use sensitivity analysis to compute a discrete shape
derivative \eqref{eq:discreteDcmL}. Next we compute a velocity corresponding
to the greatest descent direction of shape velocity subject to some regularity
constraints \eqref{eq:decent}. Finally, we use use the velocity field to
evolve the domain by moving the level-set function \eqref{eq:transport_stab}.
This procedure is then repeated in the optimization algorithm.

\begin{algorithm}[h!tb]
  \caption{Optimization algorithm}
  \label{alg:opt}
  \begin{algorithmic}[1]
  	\State Compute the velocity field/decent direction $\beta_h$ \label{alg:Start}
  	\If{$J(\Omega_{t+T})<J(\Omega_{t})$}
    \State $T = 2T$
    \Else
    \While{$J(\Omega_{t+T})>J(\Omega_{t})$}
    \State $T = T/2$
    \EndWhile
    \EndIf
    \State $t=t+T$
    \State Reinitialize $\phi_h$ to a distance function on $\Omega_{t+T}$
    \State Go back to \ref{alg:Start}
  \end{algorithmic}
\end{algorithm}


\section{Numerical Results} 
\label{sec:numerical_experiments}

\subsection{Model Problems}

For our numerical experiments we consider two model problems where we perform shape optimization using CutFEM to optimize the following designs with respect to compliance:
\begin{itemize}
\item Cantilever beam under traction load
\item L-shape beam under traction load
\end{itemize}
The design domains, boundary conditions and initial states for both these problems are described in Figure~\ref{fig:model-problems}. We assume that the material is linear elastic isotropic with a Young's modulus $E=10^4$ and a Poisson's ratio $0.3$. The traction load density in both problems is $g=[0,-20]^T$\,N/m. As objective functional we use compliance \eqref{eq:compliance-J} with a material penalty $\kappa=35$.

\begin{figure}\centering
\begin{subfigure}[b]{\textwidth}\centering
\includegraphics[width=0.75\linewidth]{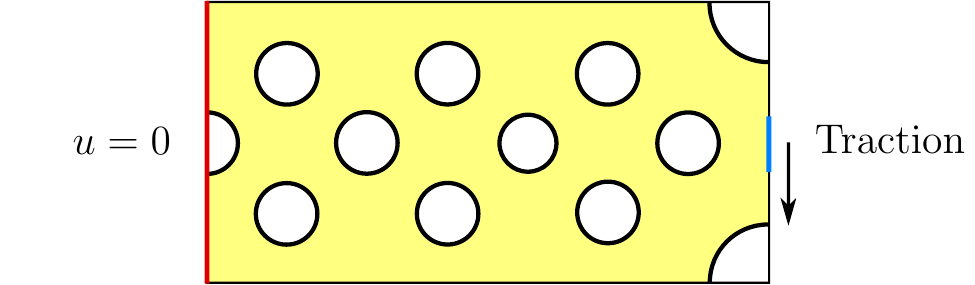}
\caption{Cantilever beam}
\label{fig:beam}
\end{subfigure}

\vspace{1.0em}
\begin{subfigure}[b]{\textwidth}\centering
\includegraphics[width=0.75\linewidth]{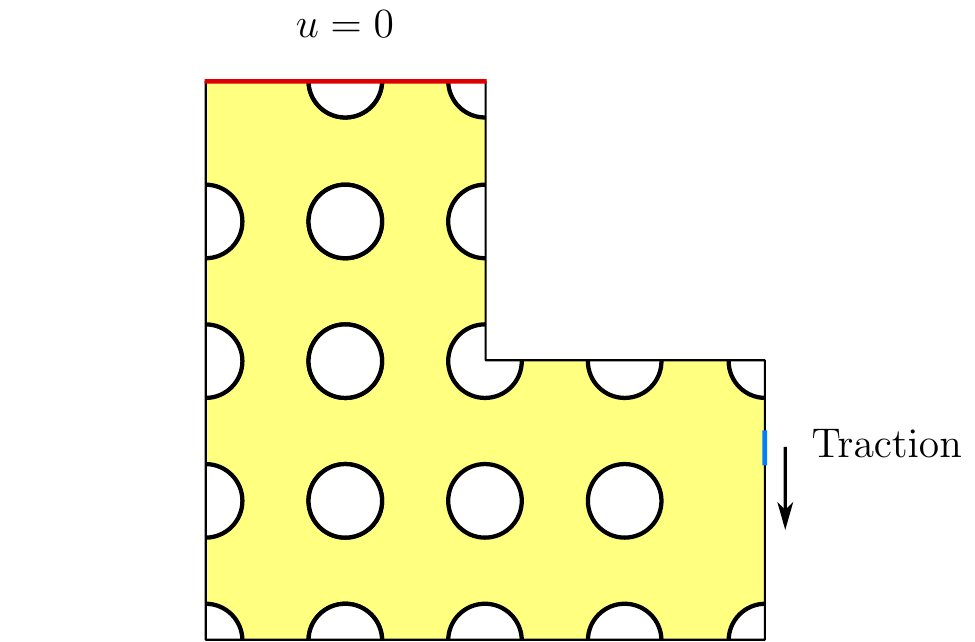}
\caption{L-shape beam}
\label{fig:L-shape}
\end{subfigure}
\caption{Schematic figures of design domains, boundary conditions and initial states used in the two model problems.
(a) \emph{Cantilever beam:} The design domain is a 2\,m $\times$ 1\,m rectangle and the traction load density on the right is evenly distributed within $\pm0.1$\,m from the horizontal center line.
(b) \emph{L-shape beam:} The design domain is a 2\,m $\times$ 2\,m square with a 1\,m $\times$ 1\,m square void in the top right corner. The traction load density on the right is evenly distributed in the interval $[5/16,1/2]$\,m from the bottom.}
\label{fig:model-problems}
\end{figure}

\subsection{Implementation Aspects}

\paragraph{Parameter Choices.}
To produce our numerical results we set $c_1=3(h/k)^2$ in \eqref{eq:b_h} and
$c_2=0.1$ in \eqref{eq:transport_stab}. Note that $h/k$ in the expression for
$c_1$ is the mesh size of the level-set mesh. We set the penalty parameter
$\gamma_D=10k^2(\mu+\lambda)$ in \eqref{eq:nitschesform} where $k$ is the
polynomial degree and $\gamma_j=10^{-7}(\mu+\lambda)$ in \eqref{eq:stab}.

\paragraph{Finite Elements.}
The CutFEM formulation works independently of the type of $H^1(\Omega_0)$ finite element and in the experiments below we use Lagrange elements of different orders $k$ on both triangles and quadrilaterals. However, a limitation with the current level-set reinitialization procedure is that it only works on $P_1$ elements and therefore we, as described in Section~\ref{sec:geom}, use $P_1$-iso-$P_k$ finite elements for the level-set representation to give good enough geometry description when using higher order elements.

\paragraph{Disconnected Domains.}
As the level-set description allow for topological changes it is possible for small parts of the domain
to become disjoint. To remedy this we use a simple filtering strategy to remove these disjoint parts which is based on properties of the direct solver. It is however possible to construct other filtering strategies instead, for example based on the connectivity of the stiffness matrix.

\subsection{Numerical Experiments}

\paragraph{Cantilever Beam.}

For the cantilever beam problem described in Figure~\ref{fig:beam}
we give results on triangles in Figure~\ref{fig:tri-geoms} and on
quadrilaterals Figure~\ref{fig:quad-geoms}. The plotted meshes
are those used for representing the solution $u$ while the mesh resolution for the
level-set description of the geometry is constant in all the examples in each figure.
Note that while the solutions are symmetric around the horizontal midline we do not impose symmetry in the algorithm.
We also give an example of the resulting stresses and displacements in Figure~\ref{fig:stresses}. We perform 50 iterations for each example.

\begin{figure}\centering
\begin{subfigure}[b]{\textwidth}\centering
\includegraphics[width=0.65\linewidth]{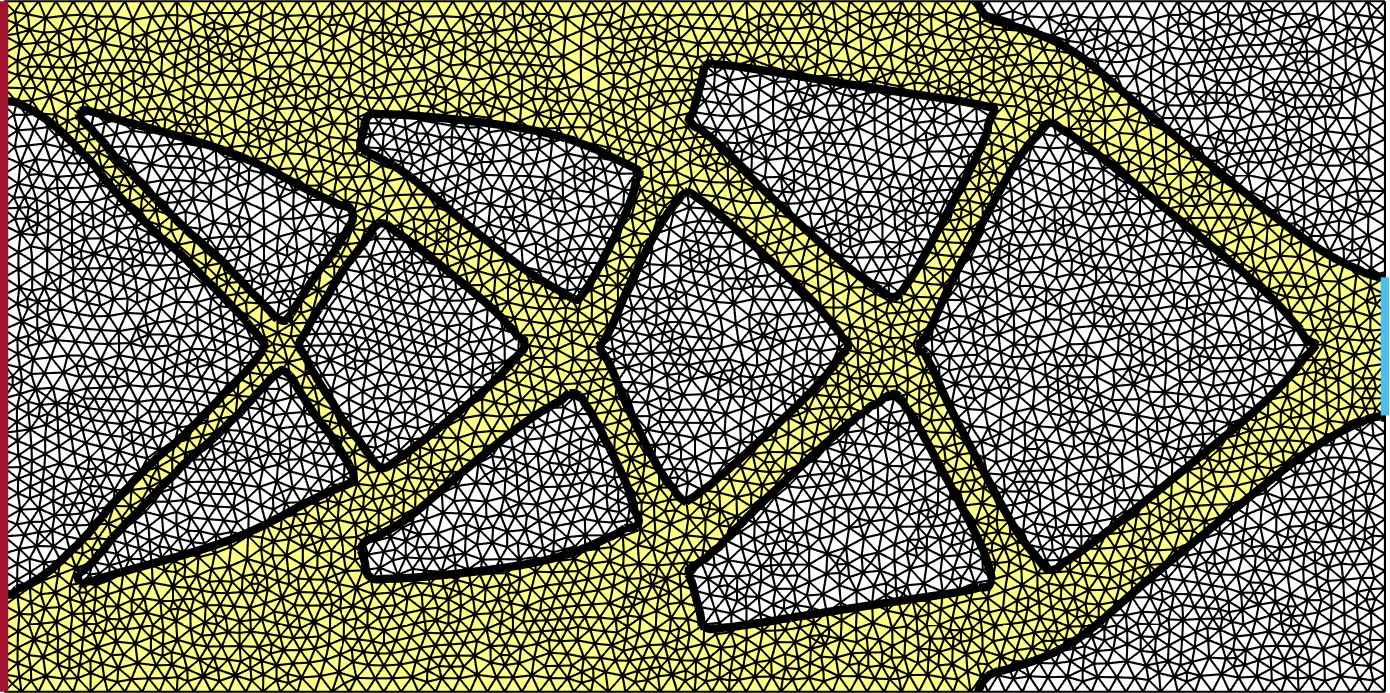}
\caption{$k=2$, $h=0.025$}
\end{subfigure}

\vspace{1em}
\begin{subfigure}[b]{\textwidth}\centering
\includegraphics[width=0.65\linewidth]{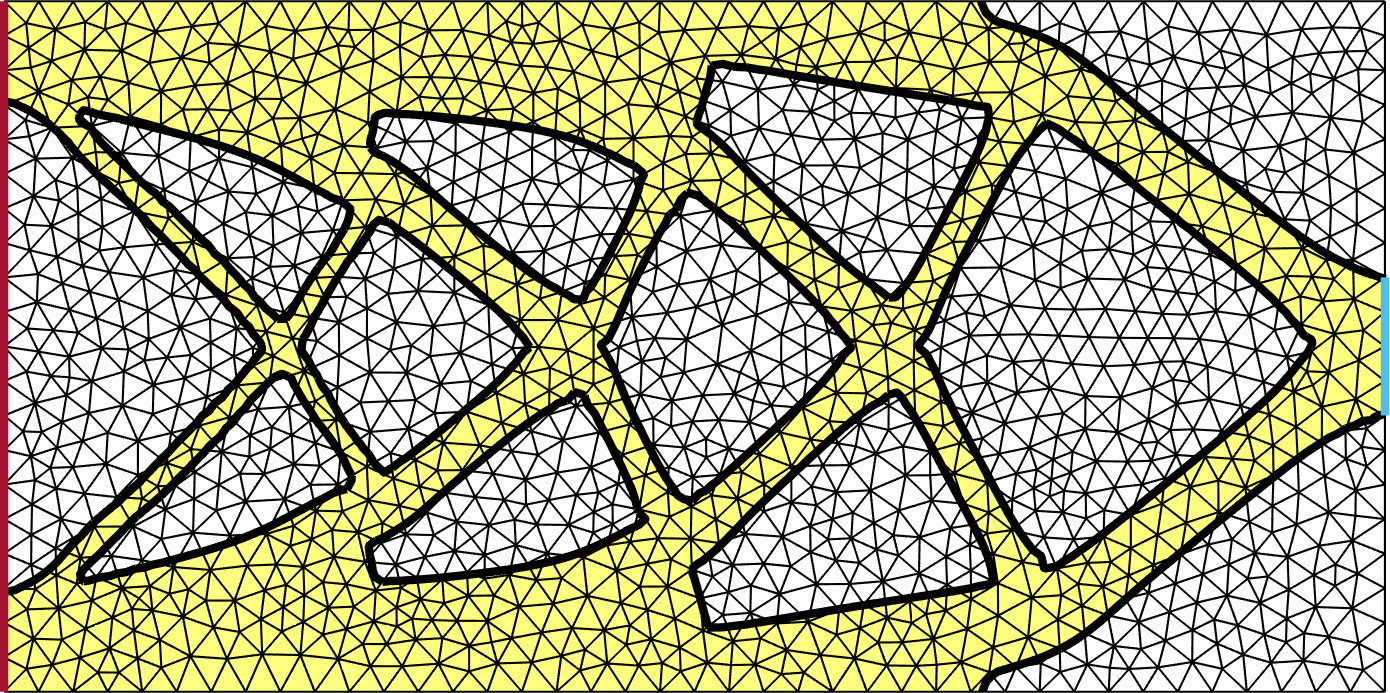}
\caption{$k=4$, $h=0.05$}
\label{fig:geom-tri-P4}
\end{subfigure}

\vspace{1em}
\begin{subfigure}[b]{\textwidth}\centering
\includegraphics[width=0.50\linewidth]{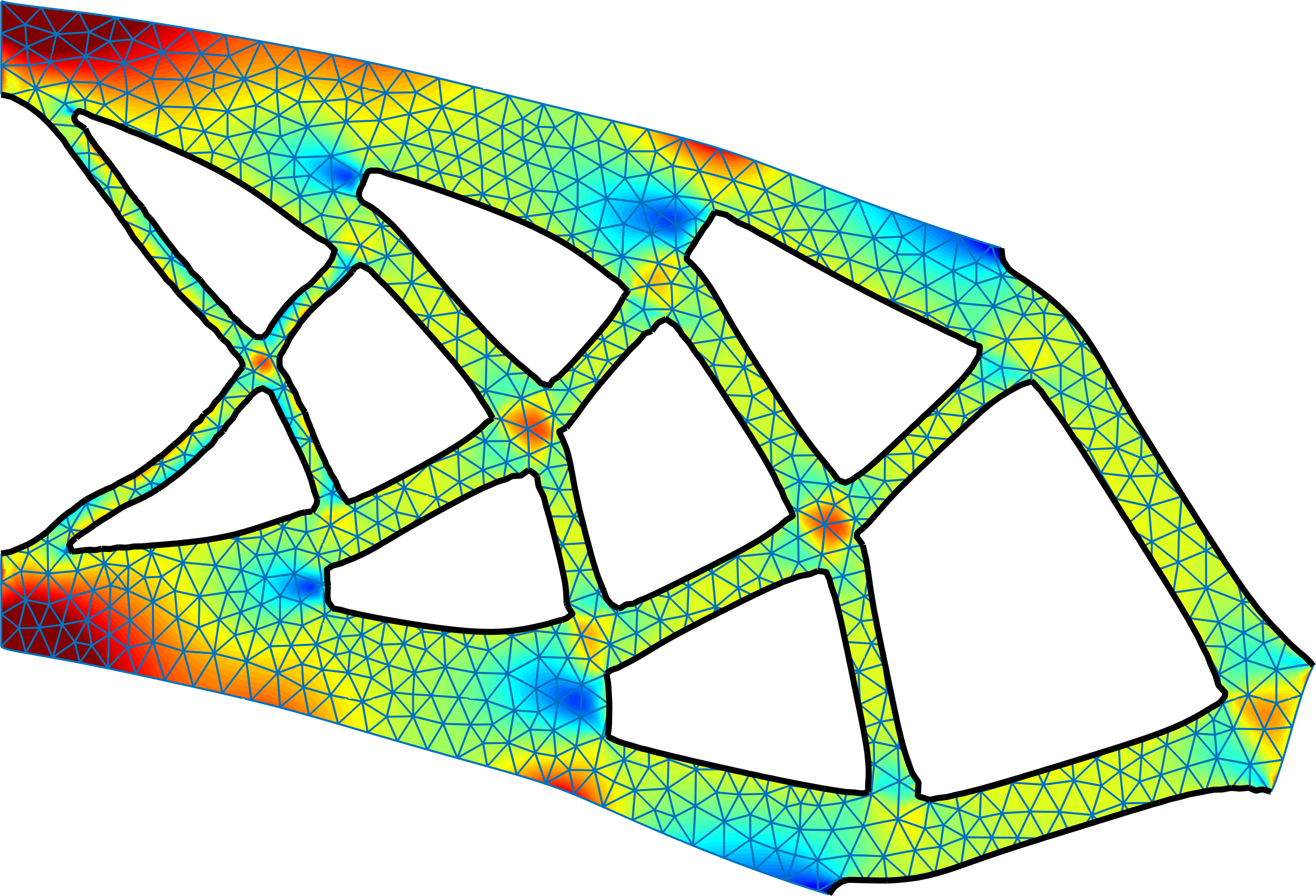}
\caption{$k=4$, $h=0.05$}
\label{fig:stresses}
\end{subfigure}

\caption{
(a)--(b) Final geometries for the cantilever beam problem using triangle elements. The mesh size for the level-set geometry description is $h/k$ and kept constant between the two examples.
(c) Visualization of (exaggerated) displacements and von-Mises stresses when using the final geometry in (b).
}
\label{fig:tri-geoms}
\end{figure}

\begin{figure}\centering
\begin{subfigure}[b]{\textwidth}\centering
\includegraphics[width=0.65\linewidth]{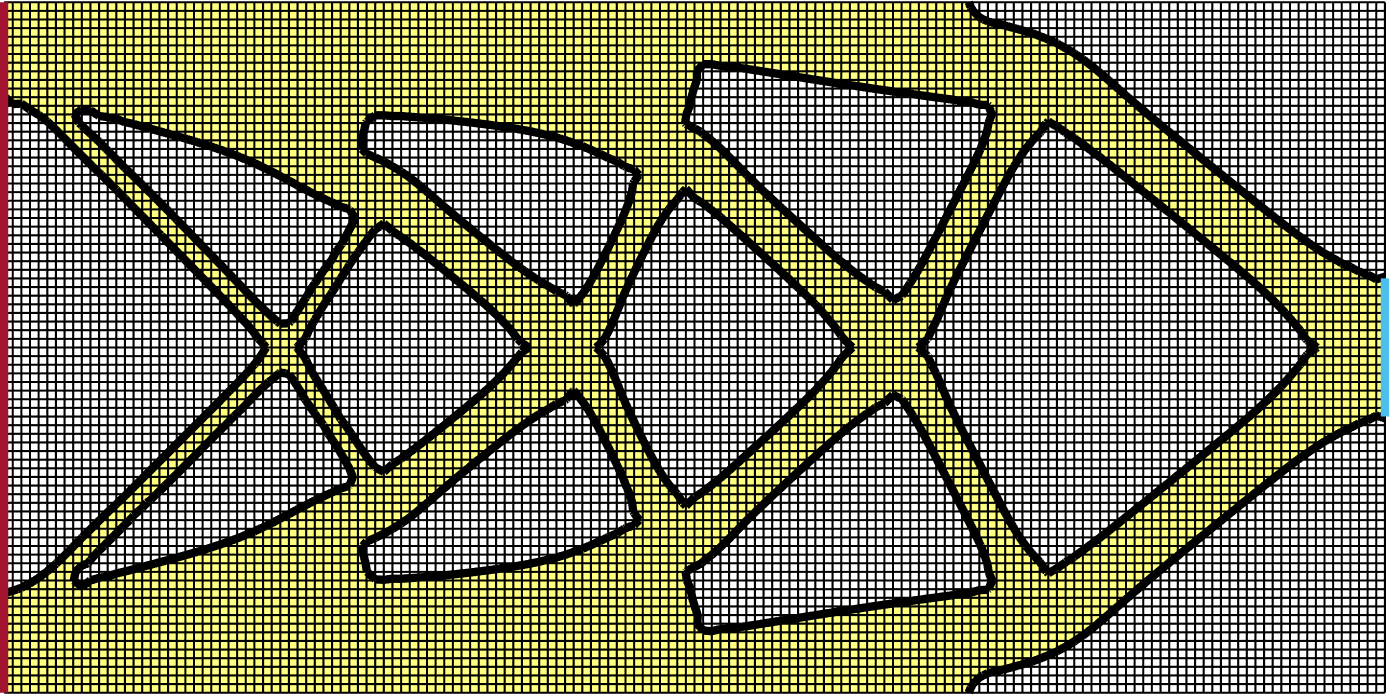}
\caption{$k=1$, $h=0.0125$}
\end{subfigure}

\vspace{1em}
\begin{subfigure}[b]{\textwidth}\centering
\includegraphics[width=0.65\linewidth]{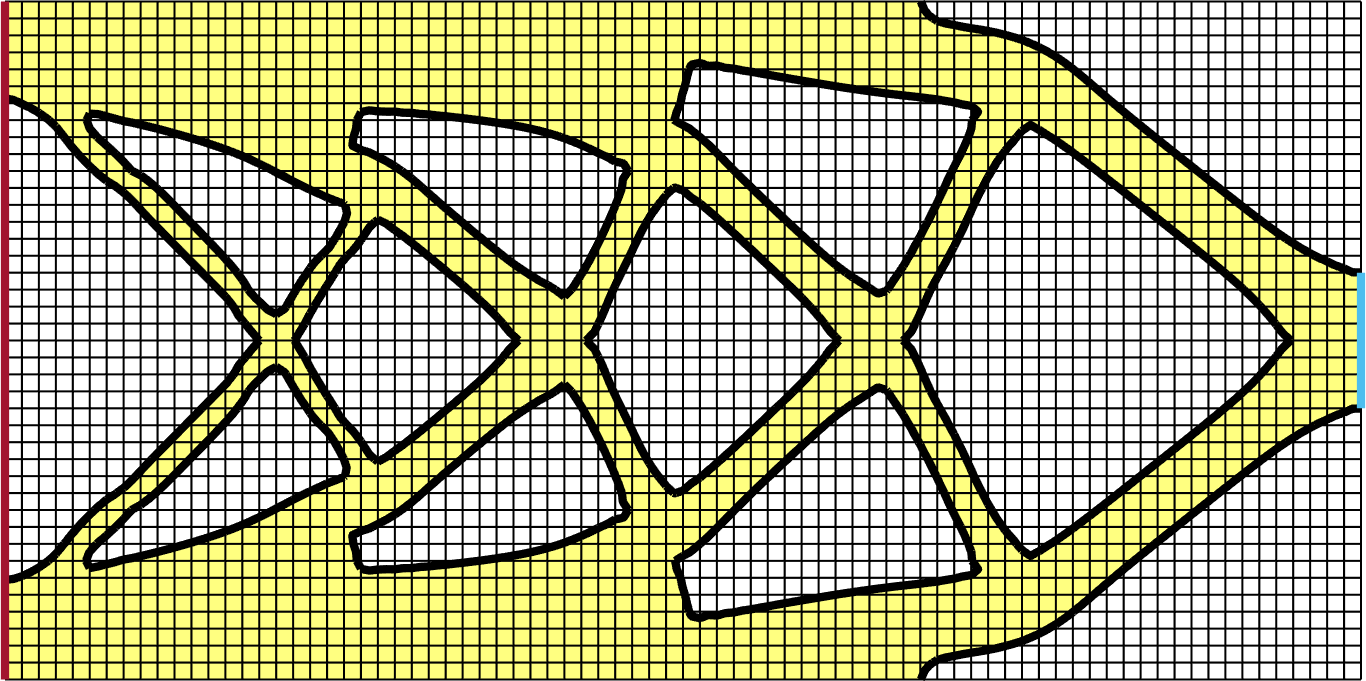}
\caption{$k=2$, $h=0.025$}
\end{subfigure}

\vspace{1em}
\begin{subfigure}[b]{\textwidth}\centering
\includegraphics[width=0.65\linewidth]{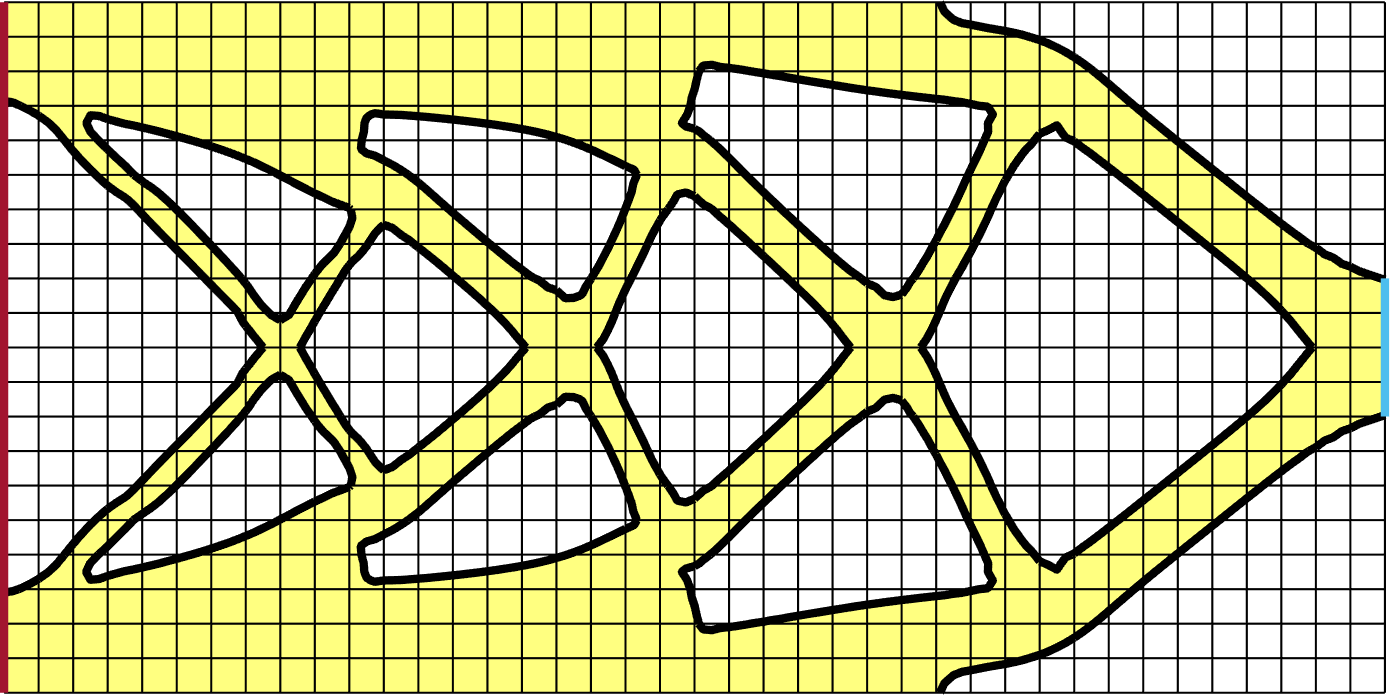}
\caption{$k=4$, $h=0.05$}
\end{subfigure}
\caption{Final geometries for the cantilever beam problem using quadrilateral elements. The mesh size for the level-set geometry description is $h/k$ and kept constant between the three examples.}
\label{fig:quad-geoms}
\end{figure}

\paragraph{L-shape Beam.}
For the L-shape beam problem described in Figure~\ref{fig:L-shape} we give results on quadratic triangles in Figure~\ref{fig:lshape-geom}. Note that in this example we actually have topological changes between the initial and final geometries. We illustrate this in the sequence presented in Figure~\ref{fig:L-shape-sequence}. We perform 50 iterations.

\begin{figure}\centering
\begin{subfigure}[b]{\textwidth}\centering
\includegraphics[width=0.65\linewidth]{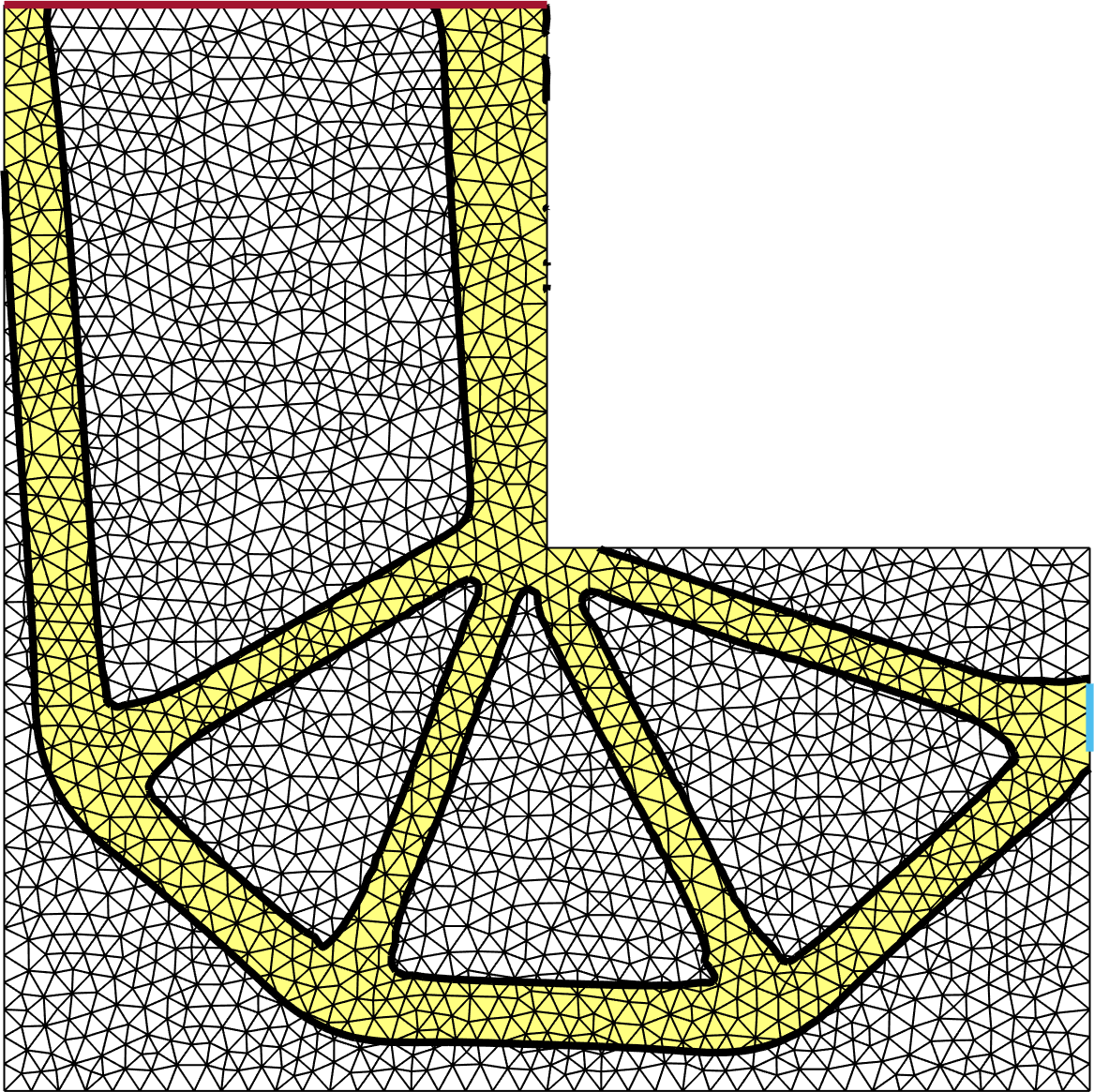}
\caption{$k=2$, $h=0.05$}
\end{subfigure}

\caption{Final geometry for the L-shape beam problem on triangles.}
\label{fig:lshape-geom}
\end{figure}

\def\colRatioOne{0.25}
\def\colRatioTwo{0.9}
\begin{figure}\centering
\begin{subfigure}[b]{\colRatioOne\textwidth}\centering
\includegraphics[width=\colRatioTwo\linewidth]{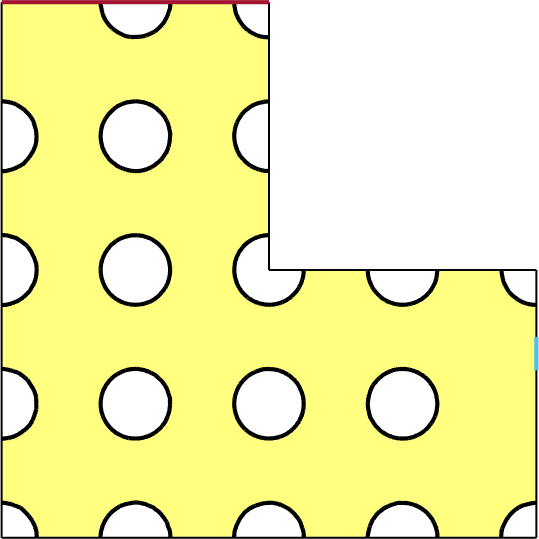}
\caption{Initial state}
\end{subfigure}
\begin{subfigure}[b]{\colRatioOne\textwidth}\centering
\includegraphics[width=\colRatioTwo\linewidth]{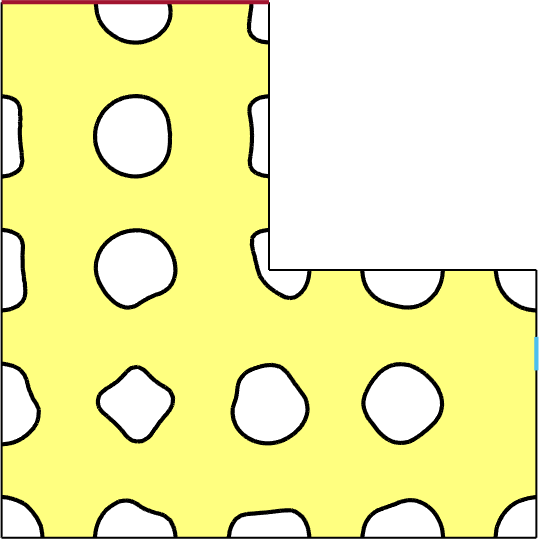}
\caption{Iteration 1}
\end{subfigure}
\begin{subfigure}[b]{\colRatioOne\textwidth}\centering
\includegraphics[width=\colRatioTwo\linewidth]{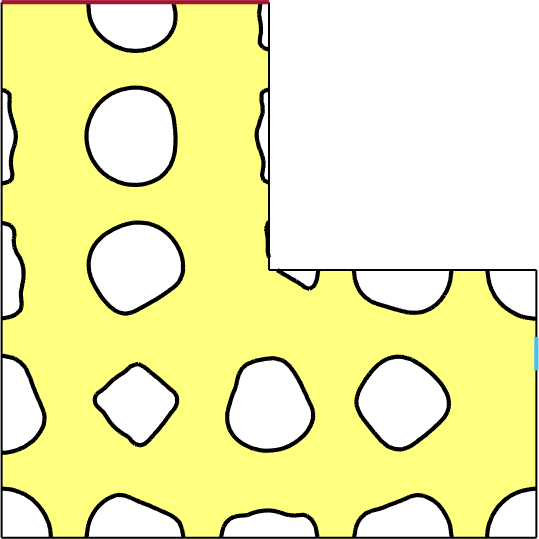}
\caption{Iteration 2}
\end{subfigure}

\vspace{1em}
\begin{subfigure}[b]{\colRatioOne\textwidth}\centering
\includegraphics[width=\colRatioTwo\linewidth]{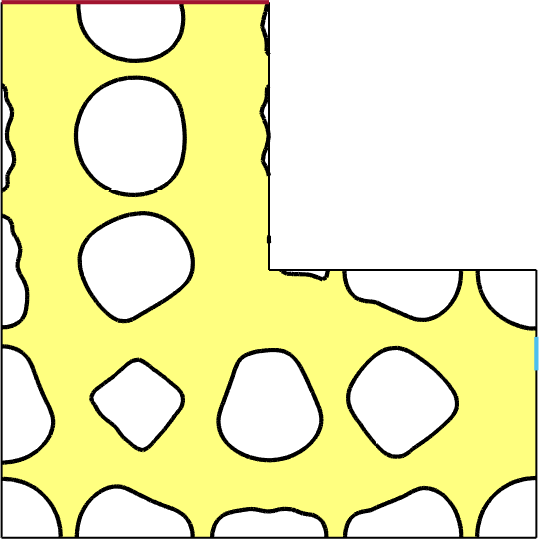}
\caption{Iteration 3}
\end{subfigure}
\begin{subfigure}[b]{\colRatioOne\textwidth}\centering
\includegraphics[width=\colRatioTwo\linewidth]{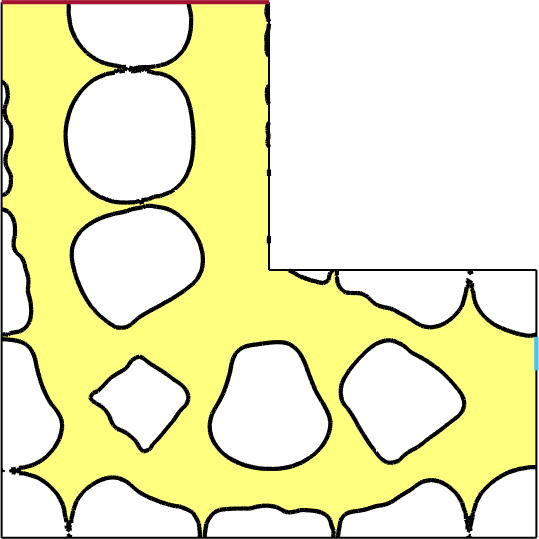}
\caption{Iteration 4}
\end{subfigure}
\begin{subfigure}[b]{\colRatioOne\textwidth}\centering
\includegraphics[width=\colRatioTwo\linewidth]{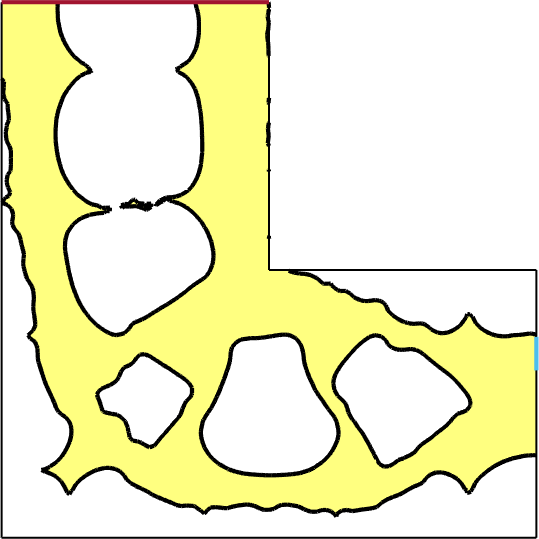}
\caption{Iteration 5}
\end{subfigure}

\vspace{1em}
\begin{subfigure}[b]{\colRatioOne\textwidth}\centering
\includegraphics[width=\colRatioTwo\linewidth]{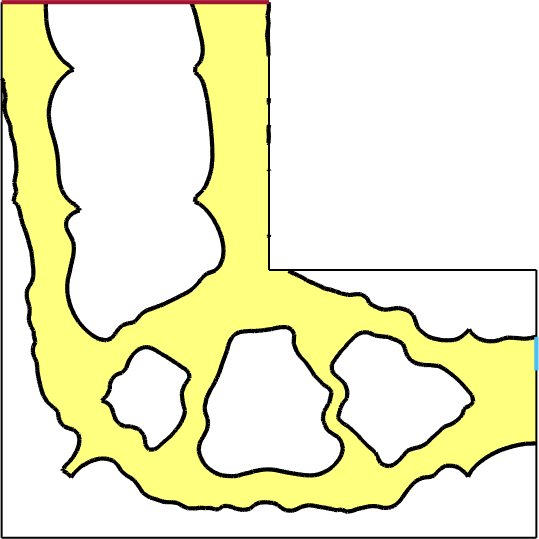}
\caption{Iteration 6}
\end{subfigure}
\begin{subfigure}[b]{\colRatioOne\textwidth}\centering
\includegraphics[width=\colRatioTwo\linewidth]{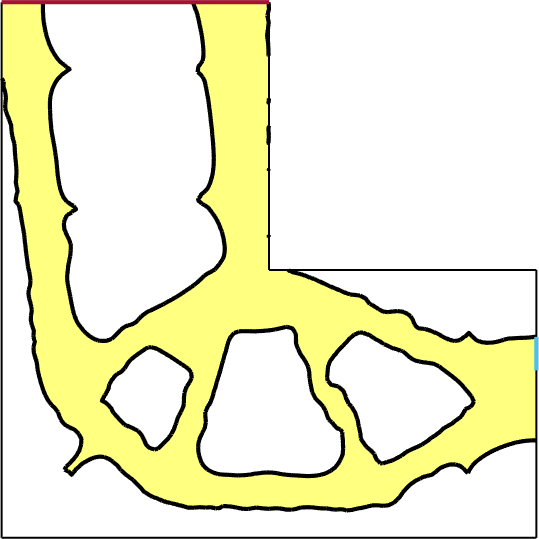}
\caption{Iteration 7}
\end{subfigure}
\begin{subfigure}[b]{\colRatioOne\textwidth}\centering
\includegraphics[width=\colRatioTwo\linewidth]{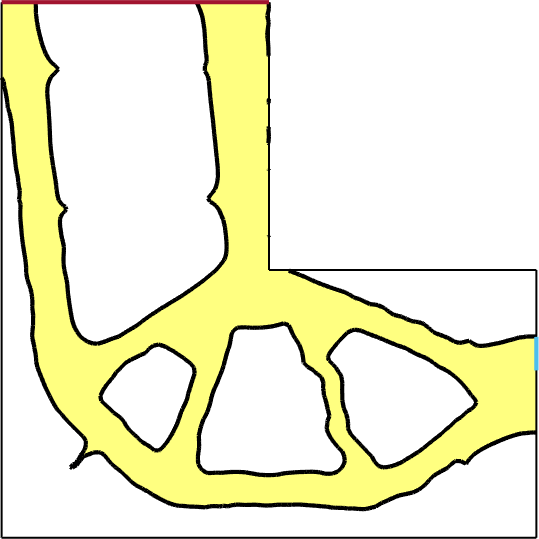}
\caption{Iteration 8}
\end{subfigure}

\vspace{1em}
\begin{subfigure}[b]{\colRatioOne\textwidth}\centering
\includegraphics[width=\colRatioTwo\linewidth]{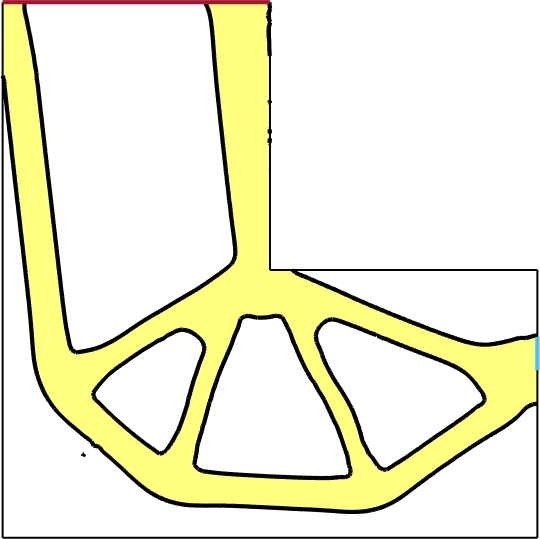}
\caption{Iteration 15}
\end{subfigure}
\begin{subfigure}[b]{\colRatioOne\textwidth}\centering
\includegraphics[width=\colRatioTwo\linewidth]{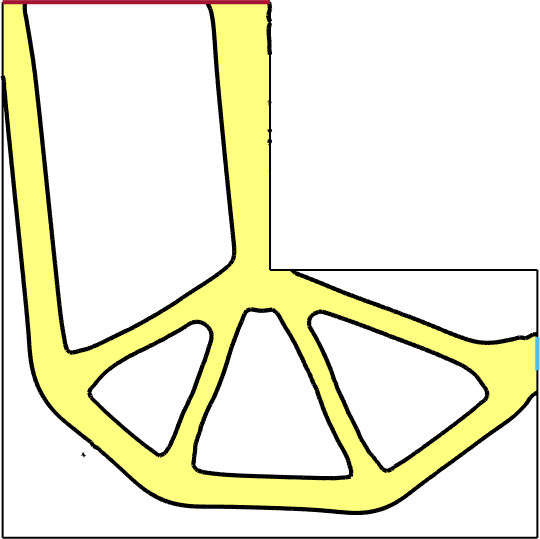}
\caption{Iteration 20}
\end{subfigure}
\begin{subfigure}[b]{\colRatioOne\textwidth}\centering
\includegraphics[width=\colRatioTwo\linewidth]{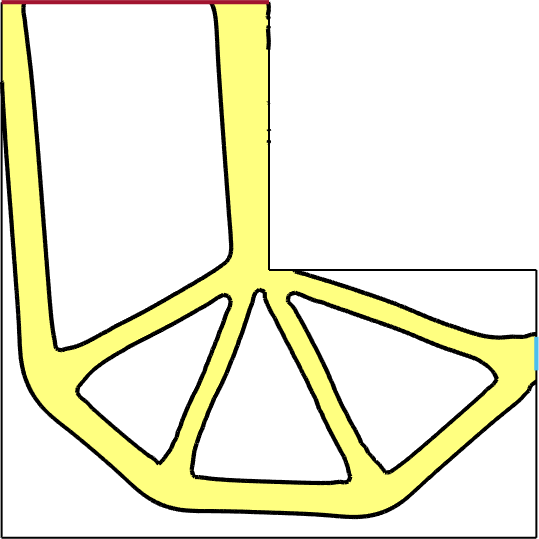}
\caption{Iteration 40}
\end{subfigure}

\caption{Sample iteration states in the L-shape beam problem ($k=2$, $h=0.05$). Note the topology changes occurring when void separations break.}
\label{fig:L-shape-sequence}
\end{figure}
\newpage 
\section{Conclusions}

We developed, implemented, and demonstrated a shape and topology 
optimization algorithm for the linear elasticity compliance problem 
based on: 
\begin{itemize}
\item A cut finite element method for linear elasticity with higher order 
polynomials on triangles and quadrilaterals.
\item A piecewise linear or bilinear level-set representation of the boundary. 
In the case of higher order polynomials we use a refined mesh for the level-set 
representation to allow a more flexible and complex geometric variation and 
utilize the additional accuracy of the higher order elements.
\item A Hamilton-Jacobi transport equation to update the geometry with 
a velocity field given by the largest descending direction of the shape 
derivative with a certain regularity requirement.  
\end{itemize}
Our numerical examples demonstrate the performance of the method and show, 
in particular, that when using higher order elements and a level-set on a 
refined mesh fine scale geometric features of thickness smaller than the 
element size can be represented and stable and accurate solutions are 
produced by the cut finite element method in each of the iterations.

Directions for future work include:
\begin{itemize}
\item Fine tuning of the local geometry based on stress measures.
\item More general design constraints.
\item Use of adaptive mesh refinement to enhance local accuracy of 
the solution and geometry representation.
\end{itemize}

\clearpage
\bibliographystyle{abbrv}
\bibliography{elasticity-optimization}

\end{document}